\documentclass[12pt,a4paper,leqno]{amsart}
\usepackage[left=25mm,right=25mm,top=35mm,bottom=38mm]{geometry}
\usepackage[OT1]{fontenc}

\usepackage{amssymb,mathtools}
\usepackage{enumitem}
\usepackage[bookmarksnumbered=true,pdfauthor={Zhi Hu and Runhong Zong},
  pdftitle={On base change of local stability in positive characteristics}]{hyperref}
\usepackage{bookmark}

\newcommand{\Q}{\mathbb Q}
\DeclareMathOperator{\Spec}{Spec}
\DeclareMathOperator{\ord}{ord}
\DeclareMathOperator{\coeff}{coeff}
\DeclareMathOperator{\mld}{mld}
\DeclareMathOperator{\Diff}{Diff}

\newcommand{\A}{\mathbb A}
\newcommand{\cO}{\mathcal O}
\newcommand{\MJ}{\mathrm{MJ}}
\numberwithin{equation}{section}

\theoremstyle{plain}
\newtheorem{theorem}{Theorem}[section]
\newtheorem{proposition}[theorem]{Proposition}
\newtheorem{lemma}[theorem]{Lemma}
\newtheorem{corollary}[theorem]{Corollary}
\theoremstyle{definition}
\newtheorem{definition}[theorem]{Definition}
\newtheorem{example}[theorem]{Example}
\newtheorem{question}[theorem]{Question}
\theoremstyle{remark}
\newtheorem{remark}[theorem]{Remark}

\allowdisplaybreaks[1]

\makeatletter
\def\ps@references{%
  \let\@oddfoot\@empty\let\@evenfoot\@empty
  \def\@evenhead{\normalfont\scriptsize\thepage\hfil}%
  \def\@oddhead{\normalfont\scriptsize\hfil\thepage}}
\makeatother

\title{On base change of local stability in positive characteristics}
\author[Zhi Hu]{Zhi Hu}
\address{Zhi Hu, School of Mathematics and Statistics, Nanjing University
of Science and Technology, Nanjing, China, 210094}
\email{halfask@mail.ustc.edu.cn}
\author[Runhong Zong]{Runhong Zong}
\address{Runhong Zong, School of Mathematics, Nanjing University,
204 Xinren Building, No.\,22 Hankou Road, Nanjing, Jiangsu, China, 210093}
\email{rzong@nju.edu.cn}
\date{November 6, 2026}
\keywords{Local stability, log canonical singularities,
Frobenius base change, wild ramification, foliations, inversion of adjunction}

\begin{document}
\begin{abstract}
Whether local stability is preserved under base change is a natural and
important question in positive characteristic.
We show that local stability, understood as log canonicity of the fiber
pair $(X,X_c)$ on the whole total space, need not be preserved under base
change in positive characteristic. For every prime $p\geq3$, we construct
a non-proper flat family of surfaces with regular total space and smooth closed
fiber whose fiber pair is purely log terminal but ceases to be log
canonical after the first Frobenius base change.

We then give an exact criterion for base-change invariance. Let
$X\to(C,c)$ be a flat family over a smooth curve germ over an
algebraically closed field of characteristic $p>0$, with geometrically
integral generic fiber. Its fiber pair remains log canonical after every
finite dominant smooth-curve base change if and only if it is log
canonical after every Frobenius power, including the original stage.
In every positive characteristic, a further nonproper family of relative
dimension $p$,
shows that the original pair and the first Frobenius pullback do not suffice. Two additional examples show why preservation of individual
discrepancies and eventual prime-to-$p$ fiber multiplicities cannot be
assumed.

For proper families with geometrically log canonical generic fiber,
we prove preservation for surface families in characteristic $p>5$ and
for locally complete intersection total spaces in every dimension and
characteristic. The existence of a proper
counterexample with a non-complete-intersection fourfold total space
remains unresolved here.
\end{abstract}
\maketitle
\tableofcontents

\section{Introduction}\label{sec:introduction}

It is natural and important to ask whether local stability is invariant
under base change in positive characteristic. This question is essential
for moduli theory: the class of families admitted by a moduli problem
must be closed under the allowed base changes. For a flat family over a
smooth curve, the fiber-pair formulation of local stability requires
that the total space, together with the scheme-theoretic fiber of
coefficient one, form a log canonical pair.
In characteristic zero, inversion of adjunction relates this condition to
the singularities of the fiber and supplies the usual base-change
compatibility; see \cite{KollarSingularities,KollarFamilies}.

In positive characteristic, the fiber-pair condition is well defined,
but its preservation under base change is a separate question. Wild
separable covers introduce a different term in the canonical divisor
formula. For purely inseparable covers, ordinary pullback of top
differential forms is unavailable, and normality of the actual fiber
product must also be checked. Throughout this paper, arbitrary base
change means arbitrary finite dominant base change by a smooth curve,
with a chosen point above the original closed point. We do not assert
a criterion for arbitrary higher-dimensional or nonreduced base schemes.

The distinction is visible in explicit examples. In every characteristic
$p\geq3$, we construct a regular affine family of surfaces with smooth
closed fiber whose fiber pair is purely log terminal, while its first
Frobenius pullback is not log canonical
(Section~\ref{sec:arbitrary}). In every characteristic $p>0$, we also
construct a family of relative dimension $p$ whose fiber pairs are log
canonical at the original and first Frobenius stages, but not at the
second (Section~\ref{sec:first}). Thus the first-step obstruction occurs
in every positive characteristic, although this construction does not
keep the relative dimension fixed. The original characteristic-two
surface example is retained as a further example. These are statements
about the whole localized total space, including centers over the generic
point of the base. These families are nonproper, their geometric generic
fibers are not log canonical, and their bad centers avoid the closed
fiber. They therefore do not give counterexamples to inversion of
adjunction along that fiber, or to a criterion restricted to a
neighborhood of it.

The main result gives a base-change-invariant test for one fixed family:
it is enough to check all powers of Frobenius. The proof works directly
with actual fiber products and treats separable covers of arbitrary
degree, without passing to a Galois closure or requiring vertical
multiplicities to become prime to the characteristic. Checking only
the original family, or only the original family and its first Frobenius
pullback, does not suffice in this generality.

Properness and geometric log canonicity of the generic fiber lead to a
different question. Under these additional hypotheses, we prove
preservation for surface families when $p>5$, and for locally complete
intersection total spaces in every dimension and characteristic. The
arguments in this final part use inversion of adjunction, unlike the
proof of the main criterion. In relative dimension three, the remaining
search concerns non-complete-intersection fourfold total spaces. We
explain the obstruction in a standard additive-group quotient family and
state a precise discrepancy condition for a possible construction.
We construct no proper counterexample satisfying all these hypotheses,
and prove no general nonexistence theorem in that dimension.

\subsection{The main theorem}

Let $k$ be an algebraically closed field of characteristic $p>0$, let
$\overline C$ be a smooth integral $k$-curve, and fix a closed point
$c\in\overline C$. We work over
\[
 C=\Spec\mathcal O_{\overline C,c}.
\]
Let $f\colon X\to C$ be flat and separated, with $X$ integral and
essentially of finite type over $k$, and assume that its generic fiber is
geometrically integral. The family is a localization of a
finite-type family. Write $D=X_c$ for the scheme-theoretic closed fiber.
A finite base change means a finite dominant morphism of smooth curves,
considered at a chosen point above $c$. The semilocal and selected-branch
conventions are specified in Section~\ref{sec:preliminaries}.

For $n\geq0$, let $C_n=C^{1/p^n}$ be the inverse Frobenius twist with
structure map $q_n\colon C_n\to C$, and put
\begin{equation}\label{eq:intro-tower}
 X_n=X\times_C C_n,\qquad D_n=(X_n)_{c_n}.
\end{equation}
These are the actual fiber products. As usual, calling a pair log
canonical includes normality of its total space and the $\Q$-Cartier
condition on its log canonical divisor.

\begin{theorem}\label{thm:main}
Under the preceding assumptions, the following conditions are equivalent.
\begin{enumerate}
\item[(i)] For every finite dominant base change
 $h\colon(C',c')\to(C,c)$ by a smooth curve, the pair
 \[
  (X',D')=(X\times_C C',X'_{c'})
 \]
 is log canonical.
\item[(ii)] The same assertion holds for every finite purely inseparable
 base change by a smooth curve.
\item[(iii)] The pair $(X_n,D_n)$ is log canonical for every $n\geq0$.
\end{enumerate}
\end{theorem}

Geometric integrality of the generic fiber ensures integrality of the
relevant tensor products of fields and of the fiber products used below.
There is no properness
assumption or dimension bound, and the Cartier index may be divisible by
$p$.

\begin{definition}\label{def:stability}
The family $X\to C$ is \emph{locally stable} if $(X,D)$ is log canonical.
It is \emph{Frobenius stable} if condition~\textup{(iii)} of
Theorem~\ref{thm:main} holds, and \emph{universally locally stable} if
condition~\textup{(i)} holds.
\end{definition}

Theorem~\ref{thm:main} identifies Frobenius stability with universal local
stability; it gives no such implication from local stability of the
original family alone. Our term ``universally locally stable'' agrees
with the base-change-invariant notion of local stability in
\cite[Definition~1.3]{HuZong}.

\begin{corollary}[Wild degree-$p$ base change]\label{cor:wild}
Suppose that all finite purely inseparable pullbacks of $X\to C$ have log
canonical fiber pairs. Then the fiber pair of the pullback by any
separable degree-$p$ smooth-curve cover is log canonical, including a
wildly ramified cover.
\end{corollary}

This is the degree-$p$ implication of \cite[Proposition~1.4]{HuZong},
with the conventions on families and fiber products specified above.
The cover need not be Galois.

\subsection{Relation to previous work}

Hu--Zong~\cite{HuZong} proposed a criterion using purely inseparable base
changes and a limiting base obtained by adjoining all $p$-power roots of
a parameter. We formulate this criterion at finite stages and replace
the normalization induction with a valuation-theoretic argument.

Benozzo--Posva~\cite[Theorem~1.0.2 and Proposition~7.2.2]{BenozzoPosva}
give a preservation implication for classes of families. Their
Terminology~7.2.1 fixes the relative dimension, and the hypothesis of
preservation under purely inseparable base change ranges over all
appropriate locally stable families. Our hypothesis concerns the
Frobenius tower of a single family. In particular, it gives no Frobenius
stability assumption for intermediate separable pullbacks; this is why
we use a direct boundedness argument. We assume geometric integrality
of the generic fiber and include no horizontal boundary, whereas
\cite{BenozzoPosva} also treats families with boundary and demi-normal
variants. Thus the two formulations have different scopes. The Frobenius
descent results of \cite[Propositions~6.1.1--6.1.2]{BenozzoPosva} contain
related discrepancy comparisons. For the quantitative inequality used
here, we apply height-one foliation adjunction; see also
\cite{PosvaPathological,PosvaQuotients}.

Posva~\cite{PosvaComparison} studies Frobenius and Artin--Schreier
base change. Example~\ref{ex:persistent} is relevant to this comparison:
even for a smooth proper product family, the multiplicity of a lifted
vertical divisor can remain divisible by $p$ at every Frobenius stage.
Our proof uses neither eventual tameness nor a normalization induction.

\subsection{Outline of the proof}

Write $A(E;V,B)=a(E;V,B)+1$ for a divisorial log discrepancy. Under the
normality, reduced-fiber, and $\Q$-Cartier hypotheses established in the
proof, a vertical divisor $E$ and its unique lift $E_1$ through one
degree-$p$ Frobenius base change satisfy
\begin{equation}\label{eq:intro-frobenius}
 A(E_1;V_1,B_1)\leq e(E_1/E)A(E;V,B).
\end{equation}
This follows from height-one foliation adjunction
(Section~\ref{sec:frobenius}) and holds even when the family being pulled
back is not log canonical.

Fix a finite separable curve cover of degree $d$, local ramification
index $b$, and different exponent $d_h$, and set
$\beta=d_h-b+1\geq0$. Let $F$ be a vertical divisor over the pullback and
$E$ its normalized restriction to the original function field. If $u$ is
a parameter on the new base and $\delta_F\geq0$ is the different exponent
on normal birational models, then
\begin{equation}\label{eq:intro-separable}
 \begin{split}
 A(F;X',D')={}&e(F/E)A(E;X,D)\\
             &+\delta_F+1-e(F/E)-\beta\ord_F(u).
 \end{split}
\end{equation}
The ramification index and the different exponent enter separately; the
formula also allows inseparable extensions of divisorial residue fields.

Suppose that $F$ has negative log discrepancy and lift it through the
Frobenius tower of the covered family. If $r_n$ denotes the multiplicity
of the base parameter at its $n$th lift and $e_n$ the successive purely
inseparable ramification index, then
\[
 e_n r_n=p r_{n+1},\qquad
 1\leq r_n\leq r_0,\qquad
 \prod_{i=0}^{n-1}e_i=\frac{p^n r_n}{r_0}.
\]
The total ramification index therefore tends to infinity, so
\eqref{eq:intro-frobenius} forces the lifted log discrepancies to tend to
minus infinity. But simultaneous Frobenius twists of the fixed separable
cover preserve $d$, $b$, and $d_h$. If all $(X_n,D_n)$ are log canonical,
\eqref{eq:intro-separable} bounds the same discrepancies below by
$1-d-\beta r_0$, a contradiction. This proves the separable case. An
arbitrary finite curve cover factors through an inverse Frobenius twist
of its target, reducing the general case to this one.

Only the positivity and integrality of the multiplicities $r_n$ are
needed; they may remain divisible by $p$. The proof requires neither
resolution of singularities nor inversion of adjunction.

\subsection{Organization}

Section~\ref{sec:preliminaries} fixes the conventions on curve germs,
fiber products and valuations. Section~\ref{sec:arbitrary} gives the
counterexamples in arbitrarily large characteristic. We then establish
the separable ramification formulas in Section~\ref{sec:separable} and
the Frobenius discrepancy inequality in Section~\ref{sec:frobenius}.
The proof of the criterion occupies Section~\ref{sec:proof}, followed
by its finite-stage formulation over the algebraic perfection in
Section~\ref{sec:perfection}. Section~\ref{sec:first} shows that a
single first Frobenius test is insufficient. The two complementary
valuation examples appear in Section~\ref{sec:examples}.
Section~\ref{sec:proper} treats proper families, the known exclusions,
and the remaining question in relative dimension three.

\medskip
\noindent\textbf{Acknowledgments.}
The authors thank Chenyang Xu for suggesting the question to the second
author, and J\'anos Koll\'ar, Jason Starr, Zhiyu Tian, Claire Voisin and
Kang Zuo for helpful comments. They acknowledge the use of OpenAI's
ChatGPT in exploring the proof of the main result and improving the
exposition. The overall mathematical ideas are due to the authors.
The authors independently verified all mathematical results and take
full responsibility for the final content.

\section{Preliminaries and conventions}\label{sec:preliminaries}

\subsection{Curve germs and actual fiber products}

For example, $\Spec k[t]_{(t)}$ is the local scheme of the affine line
at the origin, not the punctured affine line. It has a generic point
and the closed point $(t)$. Localization removes the other closed
points of the base, but does not remove centers over its generic point.

All local rings considered below are excellent. By a family over
$C=\Spec\mathcal O_{\overline C,c}$ we mean the localization of a
finite-type family. We test log canonicity using all prime divisors with
centers on this localized scheme, including those whose centers dominate
$C$, without appealing to openness of the log canonical locus.

The pullback of a finite curve cover $\overline C'\to\overline C$ to
$C$ is finite and semilocal. Fix a closed point $c'$ above $c$ and
consider the pair over $C'=\Spec\mathcal O_{\overline C',c'}$. When several
points lie above $c$, the map of the selected local rings need not be
finite. We therefore apply finite-flat arguments on the semilocal
pullback and then restrict to the chosen branch. In valuation estimates,
the degree always refers to the original function-field extension.

For any such base change we write $X'=X\times_C C'$ and $D'=X'_{c'}$.
The notation denotes the actual fiber product. We normalize only
auxiliary birational models, and state explicitly when we do so. Thus,
for a proper birational morphism $Y\to X$, the normalization of
$Y\times_C C'$ gives a normal model over $X'$. By excellence it is finite
over the fiber product, and it is proper birational over $X'$ whenever
$X'$ is normal. One may equivalently normalize over the finite semilocal
cover and then restrict to the selected branch.

\subsection{Discrepancies and divisorial valuations}

Let $V$ be a normal integral scheme essentially of finite type over $k$,
and let $B$ be an effective $\Q$-divisor such that $K_V+B$ is
$\Q$-Cartier. If $\mu:W\to V$ is a proper birational morphism from a
normal scheme carrying a prime divisor $E$, set
\begin{equation}\label{eq:discrepancies}
 a(E;V,B)=\coeff_E\bigl(K_W-\mu^*(K_V+B)\bigr),
 \qquad A(E;V,B)=a(E;V,B)+1.
\end{equation}
We choose canonical divisors compatibly. The pair is log canonical if
$A(E;V,B)\ge0$ for every prime divisor over $V$. These are the usual
conventions of \cite{KollarSingularities}; in particular, divisors on $V$
itself are included.

Every divisorial valuation is normalized to have value group $\mathbb Z$.
If $L/K$ is finite and $w$ extends a normalized divisorial valuation $v$,
its ramification index $e(w/v)$ is characterized by
\[
 w|_{K^\times}=e(w/v)v.
\]
One has $1\le e(w/v)\le[L:K]$. Divisoriality is preserved under
restriction and extension through a finite field extension. Indeed,
discreteness is preserved up to finite index and the residue-field
extension is algebraic, so the transcendence-degree criterion is unchanged.
These valuations can therefore be realized on normal birational models.
A valuation has a unique extension through a purely inseparable field
extension.

A divisor over a family is \emph{vertical} if its center maps to the
closed point of the base, and \emph{horizontal} if its center dominates
the base. For a vertical divisor $E$ and a base parameter $t$, put
\[
 m(E)=\ord_E(t)>0.
\]
This multiplicity does not depend on the choice of parameter.

\begin{lemma}\label{lem:reduced-fiber}
If $(X,D)$ is log canonical and $D$ is the scheme-theoretic fiber of a
flat family over a smooth curve germ, then $D$ is reduced. Moreover,
$K_X$ is $\Q$-Cartier.
\end{lemma}
\begin{proof}
The divisor $D$ is effective Cartier. At a generic point of a component,
log canonicity bounds its positive integral multiplicity by one. Hence $D$
is generically reduced. Since $X$ is normal, it is $S_2$; quotienting by
the nonzerodivisor defining $D$ gives an $S_1$ scheme. Thus $D$ has no
embedded associated points, and its generic reducedness implies
reducedness. Finally,
$K_X=(K_X+D)-D$ is $\Q$-Cartier.
\end{proof}

\subsection{Frobenius twists}

Let $K=k(C)$. Fix compatible purely inseparable root fields $K^{1/p^n}$,
and let $C_n$ be the normalization of $C$ in $K^{1/p^n}$. Since $k$ is
perfect and $C$ is a smooth curve germ, $C_n$ is again a smooth curve germ
and
\[
 [K^{1/p^n}:K]=p^n.
\]
If $t$ is a uniformizer on $C$, then $t_n=t^{1/p^n}$ is a uniformizer
on $C_n$. This inverse-twist notation for relative Frobenius makes the
structural maps $k$-linear and retains the coefficient twists.

Geometric integrality of the generic fiber implies that $k(X)$ and
every finite extension of $K$ are linearly disjoint over $K$. Thus the
function field of $X_n$ is the compositum $k(X)K^{1/p^n}$, of degree
$p^n$ over $k(X)$. The actual fiber products are integral as well: by
flatness, their affine coordinate rings inject into the corresponding
generic-fiber rings, which are integral. Their generic fibers remain
geometrically integral.

\begin{lemma}[Factorization of a finite curve cover]\label{lem:factorization}
If $h:C'\to C$ is a finite dominant cover of smooth curve germs over $k$,
then for some $r\ge0$ it factors, on the selected branch, as
\begin{equation}\label{eq:factorization}
 C'\xrightarrow{h_s} C_r\xrightarrow{q_r}C,
\end{equation}
where $h_s$ is separable.
\end{lemma}
\begin{proof}
Write $K=k(C)$ and $L=k(C')$, and let $r$ be the largest integer such that
$K\subseteq L^{p^r}$. Such an integer exists because a nonconstant element of the
one-variable function field $L$ cannot be a $p^r$-th power for all $r$.
Then $K^{1/p^r}$ embeds in $L$ and is not contained in $L^p$. The
separating-variable criterion for one-variable function fields over a
perfect field implies that $L/K^{1/p^r}$ is separable. Normalization of
the curves extends the field factorization to \eqref{eq:factorization}.
\end{proof}

In particular, a finite purely inseparable curve cover is, up to
isomorphism on the selected germ, an iterate of relative Frobenius.
This proves the equivalence between conditions (ii) and (iii) in
Theorem~\ref{thm:main}.

\section{Counterexamples in arbitrarily large characteristic}\label{sec:arbitrary}

We begin with the failure of the original fiber-pair condition under
a single Frobenius pullback. The relative dimension remains two as the
characteristic varies. The geometric generic fiber and the position of
the bad center will be essential when we return to proper families.

The condition considered here is log canonicity on the whole total
space, including at centers lying over the generic point of the base.
This is stronger than log canonicity on some neighborhood of the
closed fiber when the family is not proper. We use the convention
\[
 A(E;Y,B)=a(E;Y,B)+1
\]
for log discrepancy. All base changes below are actual fiber products,
without normalization.

\subsection{Construction and statement}

Let $k$ be an algebraically closed field of characteristic $p\geq3$,
and let
\[
 C=\Spec k[t]_{(t)}.
\]
Define
\begin{equation}\label{arb:eq:family}
 X=\Spec\frac{k[t]_{(t)}[x,x^{-1},y,z]}
                   {(t-x^p-y^{p+1}-z^{p+1})},
 \qquad f\colon X\longrightarrow C,
 \qquad D=(t=0).
\end{equation}
Thus $X$ is the restriction to $C$ of the family on
$\mathbb G_{m,x}\times\mathbb A^2_{y,z}$ given by
$t=x^p+y^{p+1}+z^{p+1}$.

Consider the $k$-linear Frobenius base change
\[
 C'=\Spec k[u]_{(u)}\longrightarrow C,\qquad t\longmapsto u^p,
\]
and write
\[
 X'=X\times_C C',\qquad D'=(u=0)\subset X'.
\]

\begin{proposition}\label{arb:prop:main}
The morphism $f$ is flat, separated, and surjective of relative
dimension two. Its generic fiber is geometrically integral and
geometrically normal. The total space $X$ is regular, the closed
fiber $D$ is smooth over $k$, and $(X,D)$ is purely log terminal.

The actual pullback $X'$ is normal and Gorenstein, and $D'$ is a smooth
effective Cartier divisor. There is a prime divisor $E$ over $X'$ with
\[
 \ord_E(D')=0,\qquad A(E;X',D')=1+2p-p^2<0.
\]
In particular, $(X',D')$ is not log canonical.

The morphism $f$ is not proper, and its generic fiber is not
geometrically log canonical.
\end{proposition}

\subsection{The original pair and the pullback}

Before localizing the base, eliminating $t$ in \eqref{arb:eq:family}
identifies the total space with
$\mathbb G_{m,x}\times\mathbb A^2_{y,z}$. Thus $X$ is regular.
The homomorphism $k[t]\to k[x,x^{-1},y,z]$ sending
$t$ to $x^p+y^{p+1}+z^{p+1}$ is injective. Consequently the coordinate
ring of $X$ is torsion-free over the DVR $k[t]_{(t)}$, hence flat.
The morphism is affine, so it is separated. Its generic and closed
fibers are both nonempty hypersurfaces of dimension two, which also
proves surjectivity and the assertion about relative dimension.

The closed fiber is
\begin{equation}\label{arb:eq:closed}
 D=\{x^p+y^{p+1}+z^{p+1}=0\},\qquad x\ne0.
\end{equation}
The partial derivatives with respect to $y$ and $z$ are $y^p$ and
$z^p$. If both vanished, the equation would force $x=0$.
Therefore $D$ is smooth. It is an effective Cartier divisor in the
regular scheme $X$, so $(X,D)$ is a simple normal crossing pair with
a smooth boundary, and in particular is purely log terminal.

After the base change, put $r=u-x$. Since $(u-x)^p=u^p-x^p$,
the equation becomes
\begin{equation}\label{arb:eq:surface}
 r^p=y^{p+1}+z^{p+1}.
\end{equation}
Set
\[
 S=\Spec k[r,y,z]/(r^p-y^{p+1}-z^{p+1}).
\]
Then $X'$ is a localization of $S\times\mathbb G_{m,x}$, with
base parameter $u=r+x$. More precisely, we invert every
$g(r+x)$ with $g\in k[u]$ and $g(0)\ne0$.

The element $y^{p+1}+z^{p+1}$ is not a $p$-th power in $k(y,z)$,
as its derivative with respect to $y$ is nonzero. Hence
\eqref{arb:eq:surface} is irreducible, and $S$ is integral. The
Jacobian criterion shows that its only singular point is the origin.
As a hypersurface, $S$ is Cohen--Macaulay; it is therefore normal
by Serre's criterion. It follows that $X'$ is normal. Both $X$ and
$X'$ are hypersurfaces in regular schemes, hence Gorenstein.
By flatness, $u$ is a nonzerodivisor on $X'$. Setting $u=0$ gives
exactly \eqref{arb:eq:closed}, so $D'$ is a smooth effective Cartier
divisor. In particular, $K_{X'}+D'$ is Cartier.

For later use, let $K=\overline{k(t)}$ and choose
$\alpha\in K$ with $\alpha^p=t$. In the geometric generic fiber,
the substitution $r=\alpha-x$ gives
\begin{equation}\label{arb:eq:geometric}
 X_{\bar\eta}\cong
 \{r^p=y^{p+1}+z^{p+1}\}\setminus\{r=\alpha\}
 \quad\text{over }K.
\end{equation}
The preceding irreducibility and normality arguments work over $K$.
Thus the generic fiber is geometrically integral and geometrically
normal. Since $\alpha\ne0$, the open set in \eqref{arb:eq:geometric}
contains the origin.

\subsection{A negative log discrepancy}

We first compute on the surface $S$. Blow up the origin in
$\mathbb A^3_{r,y,z}$ and take the strict transform of $S$.
Write $\pi$ for its morphism to $S$.
In the $z$-chart, write
\[
 r=zv,\qquad y=zw,\qquad F(w)=1+w^{p+1}.
\]
The strict transform has equation
\begin{equation}\label{arb:eq:chart}
 v^p=zF(w).
\end{equation}
On the open set $F(w)\ne0$, this is a smooth chart with coordinates
$v,w$, and the map to $S$ is
\begin{equation}\label{arb:eq:map}
 z=\frac{v^p}{F},\qquad
 y=\frac{v^pw}{F},\qquad
 r=\frac{v^{p+1}}{F}.
\end{equation}
The divisor $E_S=(v=0)$ on this chart is integral and reduced, and
maps to the origin of $S$. Its closure on the normalization of the
strict transform is a prime divisor on a proper birational model
of $S$. Normalization does not change this smooth chart. In
particular, $v$ is a uniformizer at the generic point of $E_S$, and
\[
 \ord_{E_S}(r)=p+1,\qquad
 \ord_{E_S}(y)=\ord_{E_S}(z)=p.
\]

Up to a nonzero constant, the hypersurface residue form
\[
 \omega=\frac{dr\wedge dz}{y^p}
\]
is a generator of the dualizing sheaf of $S$, viewed as a rational
two-form. In characteristic $p$ we have $d(v^p)=0$ and
$dF=w^p\,dw$. Differentiating \eqref{arb:eq:map} gives
\[
 dz=-\frac{v^pw^p}{F^2}\,dw,\qquad
 dr=\frac{v^p}{F}\,dv-\frac{v^{p+1}w^p}{F^2}\,dw.
\]
Consequently
\begin{equation}\label{arb:eq:residue}
 \pi^*\omega
 =\frac{-v^{2p}w^pF^{-3}}{v^{p^2}w^pF^{-p}}\,
    dv\wedge dw
 =-v^{\,2p-p^2}F^{p-3}\,dv\wedge dw.
\end{equation}
Since $F$ is a unit and $dv\wedge dw$ is a canonical generator at
the generic point of $E_S$, this proves
\begin{equation}\label{arb:eq:discrepancy}
 a(E_S;S,0)=2p-p^2,\qquad
 A(E_S;S,0)=1+2p-p^2.
\end{equation}
For $p\geq3$, we have $p(p-2)\geq3$, and hence
$1+2p-p^2\leq-2$.

The multiplicity of the exceptional divisor in this calculation
is worth making explicit. On the chart \eqref{arb:eq:chart},
$z=v^p/F$, so the exceptional Cartier divisor $(z=0)$ is $pE_S$
at the generic point of $E_S$. Adjunction for the ordinary blowup
gives coefficient $2-p$ against this Cartier divisor; the
coefficient against the prime divisor $E_S$ is therefore
$p(2-p)$, in agreement with \eqref{arb:eq:discrepancy}.

Now take the product of this birational model with
$\mathbb G_{m,x}$, and let $E$ denote the corresponding prime
divisor. Wedging \eqref{arb:eq:residue} with $dx$ shows that its
discrepancy is unchanged. Moreover,
\[
 u=r+x\equiv x\pmod E.
\]
Thus $\ord_E(u)=0$. The divisor also survives the localization
defining $X'$: for every polynomial $g(u)$ with $g(0)\ne0$, its
residue along $E$ is $g(x)\ne0$, since $x$ is transcendental over
$k$ in the residue field of $E$. All elements inverted in the
base localization therefore have valuation zero. Base changing
the proper birational model to this localization gives a
proper birational model over $X'$ on which $E$ still occurs.

Finally, $D'=(u=0)$, so
\[
 \boxed{\;
 A(E;X',D')=A(E;X',0)-\ord_E(u)
           =1+2p-p^2<0.
 \;}
\]
This proves that the first Frobenius pullback is not log canonical.
The argument uses a divisor over the actual normal fiber product;
no normalization of the fiber product is involved.

\subsection{The restrictions on this example}

The failure just proved occurs away from the closed fiber.
The center of $E$ on $X'$ is given by $r=y=z=0$, where
$u=x$ is invertible. Hence this center is disjoint from $D'$.
It lies over the generic point of $C'$.

The original morphism $f$ is not even a closed map.
Indeed, the nonempty closed subscheme
\[
 Z=V(y,z)\subset X
\]
satisfies $t=x^p$. Since $x$ is invertible, $t$ is invertible on
$Z$, and the image of $Z$ is exactly the generic point of $C$.
That singleton is not closed in $C$. In particular, $f$ is not
proper.

The geometric generic fiber is not log canonical.
By \eqref{arb:eq:geometric}, it contains the origin of the surface
\eqref{arb:eq:surface} over $\overline{k(t)}$. The same blowup and
calculation \eqref{arb:eq:residue} work over that field and give
log discrepancy $1+2p-p^2<0$. Thus geometric normality and
geometric integrality do not supply geometric log canonicity
in this example.

There is nevertheless a neighborhood of the entire closed fiber
on which the family is smooth. Namely, put $U=X\setminus Z$.
Equation \eqref{arb:eq:closed} shows that $D\subset U$, and the
relative Jacobian criterion shows that $U\to C$ is smooth:
at each point, at least one of $y^p,z^p$ is a unit.
For every finite base change by a regular curve germ, the
pullback of $U$ remains smooth over the new base, and its
closed fiber is a smooth divisor in a regular total space.
The resulting fiber pair is therefore purely log terminal.

It follows that this example does not contradict inversion of
adjunction along the closed fiber. It also gives no counterexample
for proper families or for families with geometrically log
canonical generic fiber. Properness rules out the particular
phenomenon used here: a nonempty closed subset of a scheme proper
over a local base must meet the closed fiber. Thus an open subset
containing the whole closed fiber of a proper family over $C$
must be the whole total space.

The conclusion is restricted to the stated setting: for every
$p\geq3$, a flat, nonproper family of surfaces can have a purely
log terminal fiber pair while its first Frobenius pullback has
a non-log-canonical fiber pair on the whole total space.
The dimension is fixed; only the equations depend on $p$.

\section{Separable ramification and fiber-pair discrepancies}\label{sec:separable}

Let $h:(C',c')\to(C,c)$ be a finite separable cover of degree $d$.
Choose uniformizers $t$ and $u$ at the two points and set
\begin{equation}\label{eq:base-different}
 b=\ord_{c'}(h^*t),\qquad
 d_h=\ord_{c'}\!\left(\frac{dt}{du}\right),\qquad
 \beta=d_h-b+1.
\end{equation}
Here $b$ is the ramification index and $d_h$ is the different exponent.
The integer $\beta$ is independent of the parameters and is the
coefficient of the logarithmic ramification divisor:
\[
 K_{C'}+c'=h^*(K_C+c)+\beta c'
\]
locally at $c'$. If the completed local expansion is
$t=\sum_{j\ge b}c_j u^j$, then $d_h$ is one less than the first index
$j$ with $c_j\ne0$ and $p\nmid j$. Consequently
\begin{equation}\label{eq:beta-nonnegative}
 d_h\ge b-1,\qquad \beta\ge0.
\end{equation}
In the tame case $p\nmid b$, one has $d_h=b-1$ and $\beta=0$.

\subsection{Normality and the canonical divisor}

\begin{proposition}\label{prop:separable-normality}
Let $X\to C$ be a normal flat integral family with geometrically integral
generic fiber, reduced closed fiber $D$, and $K_X$ $\Q$-Cartier. Then the
actual separable pullback $X'=X\times_C C'$ is normal and has reduced
closed fiber $D'$. With compatible canonical divisors,
\begin{equation}\label{eq:canonical-separable}
 K_{X'}+D'=g^*(K_X+D)+\beta D',\qquad g:X'\longrightarrow X,
\end{equation}
on the selected local branch. In particular, $K_{X'}+D'$ is $\Q$-Cartier.
\end{proposition}
\begin{proof}
Before localization at $c'$, the morphism $g$ is finite flat. Its
zero-dimensional fibers are Cohen--Macaulay. The flat depth formula
therefore gives $S_2$ for the pullback, using $S_2$ for $X$. At horizontal
codimension-one points, regularity is preserved because the cover is
\'etale over the generic point of $C$.

At the generic point of a component of $D$, the local ring of $X$ is a
DVR, and reducedness of the fiber makes $t$ a uniformizer. The fiber is
geometrically regular there since $k$ is perfect. Together with flatness,
this shows that $X\to C$ is smooth at that point. After base change we
obtain regularity at every vertical codimension-one point of the pullback.
The pullback is therefore $R_1$ and $S_2$, hence normal. These properties
persist after localization at $c'$. Flatness and geometric integrality of
the generic fiber give integrality of $X'$. Its selected closed fiber is
the pullback of $D$ over $k$, so it is reduced.

To compare canonical divisors, it suffices to work in codimension one.
At a generic
point of $D$ choose separating fiber coordinates $x_2,\ldots,x_N$.
Pulling back
\[
 dt\wedge dx_2\wedge\cdots\wedge dx_N
\]
introduces the factor $dt/du$, of order $d_h$ along each component of
$D'$. There is no ramification contribution at horizontal
codimension-one points. Hence
\[
 K_{X'}=g^*K_X+d_hD',\qquad g^*D=bD'
\]
on the selected branch. Adding $D'$ proves
\eqref{eq:canonical-separable}, whose right side is $\Q$-Cartier.
\end{proof}

\subsection{Comparison on birational models}

\begin{proposition}\label{prop:separable-comparison}
In the setting of Proposition~\ref{prop:separable-normality}, let $F$ be
a vertical prime divisor over $X'$ and $E$ its restricted normalized
divisorial valuation over $X$. Put
\[
 e=e(F/E),\qquad m=\ord_F(u),
\]
and let $\delta_F$ be the different exponent of the corresponding
extension of divisorial DVRs, understood by localization from their
finite semilocal normalization. Then $\delta_F\ge0$ and
\begin{equation}\label{eq:separable-discrepancy}
 A(F;X',D')=eA(E;X,D)+\delta_F+1-e-\beta m.
\end{equation}
If $(X,D)$ is log canonical, this gives
\begin{equation}\label{eq:separable-lower-bound}
 A(F;X',D')\ge1-d-\beta m.
\end{equation}
\end{proposition}
\begin{proof}
Choose a normal proper birational model $\mu:Y\to X$ carrying $E$.
Let $Y'$ be the normalization of the integral fiber product
$Y\times_C C'$. Then $\mu':Y'\to X'$ is proper birational, and $Y'$
carries $F$: normalization is finite by excellence, and localizing the
normalization of the divisorial DVR at $E$ realizes the extensions of its
valuation. Equivalently, one may normalize over the finite semilocal
curve cover and then restrict to the selected $c'$-branch. Finiteness is
used before this restriction; the map of the selected local DVRs need
not be module-finite.

Write $g_Y:Y'\to Y$ for the induced morphism. Near the generic point of
$F$, the canonical comparison is
\[
 K_{Y'}=g_Y^*K_Y+\delta_F F.
\]
As an identity in codimension one, this requires no global
$\Q$-Cartier assumption on $K_Y$. The coefficient $\delta_F$ is the
order of the nonzero determinant of the pullback map on top differential
forms. The divisorial local rings of the normal models are essentially
smooth over the perfect field $k$, so we may calculate on smooth
finite-type neighborhoods before localization. There the different is
cut out by this determinant
\cite[Lemma~49.12.6, Tag~\href{https://stacks.math.columbia.edu/tag/0BWJ}{0BWJ}]{Stacks}.
Localizing at $F$
identifies the coefficient with the different exponent and shows that
it is nonnegative, even when $k(F)/k(E)$ is inseparable.

Pulling back \eqref{eq:canonical-separable} to $Y'$ and comparing
coefficients at $F$, we obtain from the commutative diagram with
$\mu,\mu',g,g_Y$ the equality
\[
 a(F;X',D')=e\,a(E;X,D)+\delta_F-\beta m.
\]
Adding one proves \eqref{eq:separable-discrepancy}. Since $1\le e\le d$,
nonnegativity of $A(E;X,D)$ and $\delta_F$ proves
\eqref{eq:separable-lower-bound}.
\end{proof}

In \eqref{eq:separable-discrepancy}, the ramification index and the
different exponent give distinct contributions, which cannot be
interchanged in the wild case.

\subsection{Tame covers and coprime multiplicities}

\begin{corollary}[Tame invariance]\label{cor:tame}
In the setting of Proposition~\ref{prop:separable-normality}, assume
that $h$ is tamely ramified at the selected point. Then
\[
 A(F;X',D')=e(F/E)A(E;X,D)
\]
for corresponding vertical divisors. Consequently $(X',D')$ is log
canonical if and only if $(X,D)$ is log canonical.
\end{corollary}
\begin{proof}
After an \'etale localization, the curve map has the form $t=u^b$ with
$p\nmid b$. Tame extensions of DVRs remain tame after base change and
normalization; see
\cite[Lemma~15.116.9, Tag~\href{https://stacks.math.columbia.edu/tag/0EXY}{0EXY}]{Stacks}.
Hence the induced extension
at $F$ has different exponent $\delta_F=e(F/E)-1$. Since $\beta=0$,
the formula follows from Proposition~\ref{prop:separable-comparison}.
Every vertical divisorial valuation downstairs has an extension on the
chosen branch, as one sees by passing to henselian curve germs and using
the displayed form of the tame cover. For horizontal divisors, \'etale
invariance of log canonicity over the generic point of the base gives
the same conclusion. The asserted equivalence follows.
\end{proof}

Suppose next that $h$ has degree $p$ and is wildly totally ramified at
$c'$. A completed local parameter can be chosen so that
\begin{equation}\label{eq:wild-normal-form}
 t=u^p+v u^{p+s}+\text{higher powers of }u,
 \qquad v\in k^\times,\quad s>0,\quad p\nmid s.
\end{equation}
Absorb the initial terms whose exponents are divisible by $p$ into a
new parameter. Separability then provides a first term with exponent
not divisible by $p$. Any such parameter will suffice.
Equation~\eqref{eq:base-different} gives
\begin{equation}\label{eq:wild-beta}
 d_h=p+s-1,\qquad\beta=s.
\end{equation}
Thus $s$ denotes the logarithmic ramification coefficient, not the
different exponent.

\begin{proposition}[The coprime-multiplicity case]\label{prop:coprime}
In the setting of Proposition~\ref{prop:separable-normality}, assume
that $h$ is a separable degree-$p$ cover wildly totally ramified at
$c'$. If $E$ is vertical and $M=\ord_E(t)$ is prime to $p$, then for
every extension $F$ one has
\begin{equation}\label{eq:coprime-data}
 e(F/E)=p,\qquad\ord_F(u)=M,\qquad\delta_F=\beta M+p-1,
\end{equation}
and therefore
\[
 A(F;X',D')=pA(E;X,D).
\]
\end{proposition}
\begin{proof}
The valuation relation $p\ord_F(u)=e(F/E)M$, together with
$e(F/E)\le p$ and $p\nmid M$, forces the first two equalities in
\eqref{eq:coprime-data}. By the fundamental inequality for extensions of
valuations, the residue-field extension is trivial.

Near the generic point of $E$, choose a uniformizer $x$ and separating
residue coordinates $y_2,\ldots,y_N$, and write $t=x^Mf$ with $f$ a
unit. Let $z$ be a uniformizer upstairs. The same $y_i$ are separating
residue coordinates at $F$. Wedge
\[
 h'(u)\,du=d(x^Mf)
\]
with $dy_2\wedge\cdots\wedge dy_N$. Downstairs the coefficient of
$dx$ is $x^{M-1}(Mf+x\partial_x f)$, whose parenthetical factor is a
unit. Upstairs $u=z^M$ times a unit, so
$du\wedge dy_2\wedge\cdots\wedge dy_N$ has order $M-1$. Comparing
orders gives
\[
 M d_h+(M-1)=p(M-1)+\delta_F.
\]
This is the third equality in \eqref{eq:coprime-data}. Substitution into
\eqref{eq:separable-discrepancy} completes the proof.
\end{proof}

Under the same hypothesis that the base cover is totally ramified of
degree $p$, the general formula is
\begin{equation}\label{eq:degree-p-general}
 A(F;X',D')=eA(E;X,D)+\delta_F+1-e-\beta eM/p.
\end{equation}
Two useful specializations are
\begin{align}
 A(F;X',D')&=pA(E;X,D)+\delta_F+1-p-\beta M,
 &&e(F/E)=p,\label{eq:degree-p-equals-p}\\
 A(F;X',D')&=A(E;X,D)+\delta_F-\beta M/p,
 &&e(F/E)=1.\label{eq:degree-p-equals-one}
\end{align}
In the second case $p\mid M$. For a non-Galois degree-$p$ cover,
extensions of a divisorial valuation can also have intermediate
ramification indices; \eqref{eq:degree-p-general} covers these as well.
The factor $1/p$ in \eqref{eq:degree-p-equals-one} records the fiber
multiplicity and is essential. When $p\mid M$, the original discrepancy
need not be preserved, as Example~\ref{ex:discrepancies} shows. We treat
this case by Frobenius pullback, without further normal-form reductions.

\section{Frobenius and the growth of negative discrepancies}
\label{sec:frobenius}

\subsection{The quotient derivation and adjunction}

A $1$-foliation on a normal variety $V$ over the
perfect field $k$ is a coherent saturated subsheaf
$\mathcal G\subset T_{V/k}$ closed under Lie brackets and $p$-th
iterates of derivations; saturation means that $T_{V/k}/\mathcal G$
is torsion-free. The functions annihilated by $\mathcal G$ define a
normal finite purely inseparable quotient $V\to V/\mathcal G$.
The quotient has height one: $k(V)^p\subseteq k(V/\mathcal G)$,
and the relative Frobenius $F_{V/k}$ factors through it. The quotient
of a rank-$r$ foliation has degree $p^r$; we shall use rank one and degree
$p$. For the Jacobson correspondence, see
\cite[Sections~2.4--2.5, especially Theorem~2.5.12]{PosvaQuotients}.

Let $Z\to (T,0)$ be a flat integral family over a smooth curve germ
with geometrically integral generic fiber. Consider the Frobenius base change
\[
  q\colon Z_1=Z\times_T T_1\longrightarrow Z,
  \qquad t=s^p,
\]
and let $B$ and $B_1$ denote the scheme-theoretic closed fibers. In this
section, $Z$ and $Z_1$ are assumed normal, $B$ and $B_1$
reduced, and $K_Z$ and $K_{Z_1}$ $\Q$-Cartier. No log canonicity
assumption is imposed on either fiber pair.

Write $T=\Spec R$ and $T_1=\Spec S$, with both rings embedded in the
compatible root fields fixed earlier. We have
\[
  S^p=R.
\]
The $p$-th power of any element of $S$ lies in the fraction
field of $R$ and is integral over $R$, so normality places it in $R$.
Conversely, the $p$-th root of an element of $R$ is
integral over $R$ and therefore lies in $S$. Finiteness of Frobenius for
smooth curve germs over the perfect field $k$ makes $S$ a finite
torsion-free $R$-module of rank $p$, hence a free module. Both
residue fields are $k$, so the ramification index equals $p$.
Thus, for any uniformizer $s$ of $S$, the element $t=s^p$ is a
uniformizer of $R$.

Since the images of $1,s,\ldots,s^{p-1}$ form a $k$-basis of
$S/tS=S/(s^p)$, Nakayama's lemma and the rank computation give
\begin{equation}
  S=\bigoplus_{i=0}^{p-1}Rs^i
    \simeq R[U]/(U^p-t).
  \label{eq:frob-free-basis}
\end{equation}
Differentiation with respect to $s$ therefore defines a regular
$R$-derivation on $S$ satisfying
\begin{equation}
  \partial_s(s)=1,
  \qquad \partial_s^{[p]}=0,
  \qquad \ker(\partial_s\colon S\to S)=R.
  \label{eq:frob-base-derivation}
\end{equation}
Here $\partial_s^{[p]}$ denotes the $p$-fold composite. The last two
identities follow by differentiating the expansion in
\eqref{eq:frob-free-basis}, since $1,\ldots,p-1$ are
units in $k$.

On each affine open $\Spec A\subset Z$, the same formula on
$A\otimes_R S=\bigoplus_{i=0}^{p-1}As^i$ defines a derivation
annihilating $A$ and having invariant ring $A$. Compatibility with
restriction gives a global
regular derivation $\partial_s$ on the actual fiber product $Z_1$.
Its saturated rank-one foliation is
\[
  \mathcal F=\mathcal O_{Z_1}\partial_s.
\]
For if a rational multiple $a\partial_s$ is regular, its value on
the globally regular function $s$ shows that $a$ is regular.
The morphism $q$ is therefore the height-one quotient by $\mathcal F$,
and $K_{\mathcal F}$ is Cartier, trivialized by the dual rational
section $\partial_s^*$.

For a normal height-one quotient $\pi\colon V\to W$ defined by a
saturated foliation $\mathcal G$, we use the convention
$\mathcal O_V(K_{\mathcal G})=(\det\mathcal G)^{-1}$ as rank-one
reflexive sheaves. The adjunction formula is
\begin{equation}
  \pi^*K_W\sim K_V+(p-1)K_{\mathcal G}.
  \label{eq:frob-adjunction}
\end{equation}
The inseparable ramification formula~\eqref{eq:frob-adjunction}
comes from taking determinants in the tangent-sheaf sequence for a
height-one quotient, which gives the factor $p-1$;
see \cite[Proposition~2.5.14]{PosvaQuotients}.
For a nontrivial purely inseparable map, ordinary pullback of top
differential forms vanishes, so the comparison is made directly
between canonical sheaves. The foliation is trivial on $Z_1$;
its saturated transform on a birational model may have nonzero
discrepancy, denoted by $\gamma$ below.

Choose compatible rational sections in this determinant-sheaf
comparison on the original spaces and on their birational models.
In codimension one they give an equality of Weil divisors;
subtracting the pullback of the equality downstairs gives the
discrepancy formula. Linear equivalence alone would not suffice
for this coefficient comparison. For adjunction and its discrepancy
form, see
\cite[Proposition~4.2.3 and the proof of Theorem~4.2.5]{PosvaQuotients}.
The normalized birational pullback and the discrepancy comparison
are also stated in
\cite[Proposition~2.9 and Theorem~2.10(2)--(3)]{PosvaPathological}.
For families under Frobenius base change, see
\cite[Proposition~6.1.1]{BenozzoPosva}.

\subsection{A valuationwise inequality}

\begin{proposition}
\label{prop:frob-ineq}
Let $E$ be a vertical prime divisor over $Z$, with unique
lift $E_1$ to $k(Z_1)$, and put $e=e(E_1/E)\in\{1,p\}$. Then
\begin{equation}
  A(E_1;Z_1,B_1)\le eA(E;Z,B).
  \label{eq:frob-inequality}
\end{equation}
\end{proposition}

\begin{proof}
Choose a proper birational morphism $\mu\colon Y\to Z$ from a
normal scheme carrying $E$, and let $Y_1$ be the normalization of
$Y\times_Z Z_1$ in $k(Z_1)$, or equivalently the normalization
of $Y$ in $k(Z_1)$. Since $Y$ is integral and dominates the curve
germ, it is flat over the base DVR. Its generic fiber is birational
to the geometrically integral generic fiber of $Z$, and is therefore
geometrically integral as well. Hence the actual pullback is integral,
and excellence ensures that its normalization is finite. This gives
the commutative diagram
\[
\begin{array}{ccc}
 Y_1 & \xrightarrow{\ q_Y\ } & Y \\
 {\scriptstyle\mu_1}\downarrow && \downarrow{\scriptstyle\mu} \\
 Z_1 & \xrightarrow{\ q\ } & Z,
\end{array}
\]
where $q_Y$ is finite and $\mu_1$ is proper and birational.
Pure inseparability gives a unique extension of the valuation of $E$,
represented by a prime divisor $E_1$ on $Y_1$ above $E$.
The displayed morphism gives its center on $Z_1$.
The saturation on $Y_1$ of the rational pullback of $\mathcal F$
has quotient $Y$; see
\cite[Theorem~2.10(2)]{PosvaPathological}.

Let $\mathcal V=\mathcal O_{Y_1,E_1}$ and choose a uniformizer
$z\in\mathcal V$ of this DVR. The verticality of $E$ implies that
$E_1$ is vertical, so
\[
  s=z^m v,
  \qquad m=\ord_{E_1}(s)>0,
  \qquad v\in\mathcal V^{\times}.
\]
Normality at this codimension-one point and the perfectness of
$k$ make the DVR $\mathcal V$ essentially smooth over $k$.
Its module of regular $k$-derivations is therefore free. Clear the
poles of $\partial_s$ and remove the common power of $z$ from its
coefficients. This gives an integer $\gamma$ and a primitive regular
generator $\psi$ of the saturated rank-one foliation with
\begin{equation}
  \partial_s=z^{-\gamma}\psi.
  \label{eq:frob-generator}
\end{equation}
The exponent is unique, and changing the uniformizer changes the
generator by a unit. With our sign convention, the foliation
discrepancy along $E_1$ is $\gamma$, since the dual rational section
induced by $\partial_s$ satisfies
\[
  \partial_s^*(\psi)=z^{\gamma}.
\]
Apply the compatible determinant comparison
\eqref{eq:frob-adjunction} to $q$ and $q_Y$. The original foliation
has the nowhere vanishing generator $\partial_s$, so its canonical
divisor represented by $\partial_s^*$ is zero. Subtract the
pulled-back equality on $Z_1$ from the equality on $Y_1$ and take
the coefficient of $E_1$. We obtain
\begin{equation}
  a(E_1;Z_1,0)
    =e\,a(E;Z,0)-(p-1)\gamma.
  \label{eq:frob-discrepancy}
\end{equation}
Equivalently, apply \cite[Theorem~2.10(3)]{PosvaPathological} with
ramification index $e$ and foliation discrepancy $\gamma$.
The factor $e$ is the multiplicity in the finite pullback of the
downstairs prime divisor.

From $t=s^p$ we have $e\,\ord_E(t)=pm$. The fibers are reduced
Cartier divisors, with pullback coefficients $\ord_E(t)$ and $m$
on the respective birational models. Subtracting the fiber
contributions in \eqref{eq:frob-discrepancy} and adding one gives
\begin{equation}
  A(E_1;Z_1,B_1)
   =eA(E;Z,B)+1-e+(p-1)(m-\gamma).
  \label{eq:frob-log-discrepancy}
\end{equation}

By \eqref{eq:frob-base-derivation} and
\eqref{eq:frob-generator},
\begin{equation}
  z^{\gamma}=\psi(s)
     =m z^{m-1}v\psi(z)+z^m\psi(v).
  \label{eq:frob-generator-estimate}
\end{equation}
For $e=p$, regularity of $\psi$ gives $\gamma\ge m-1$, so
\[
  1-p+(p-1)(m-\gamma)\le0.
\]
For $e=1$, a uniformizer $w$ of $\mathcal O_{Y,E}$ has valuation
one in $\mathcal V$. Write $w=zv_0$ with
$v_0\in\mathcal V^{\times}$. Since $\psi$ annihilates the
downstairs function field,
\[
  0=\psi(w)=v_0\psi(z)+z\psi(v_0).
\]
It follows that $\psi(z)\in(z)$, whence
\eqref{eq:frob-generator-estimate} gives $\gamma\ge m$ and
$1-e+(p-1)(m-\gamma)\le0$. In both cases,
\eqref{eq:frob-log-discrepancy} proves the inequality.
If $p\mid m$, the first term on the right of
\eqref{eq:frob-generator-estimate} vanishes, leaving the estimates
unchanged.
\end{proof}

\begin{corollary}[Descent for a vertical divisor]
\label{cor:frob-descent}
In the setting of Proposition~\ref{prop:frob-ineq}, one has
\[
 A(E_1;Z_1,B_1)\ge0\quad\Longrightarrow\quad A(E;Z,B)\ge0.
\]
\end{corollary}

\begin{proof}
The assertion follows from \eqref{eq:frob-inequality}, since $e>0$.
\end{proof}

We shall use the bound itself in the proof of the main theorem.
Proposition~\ref{prop:frob-ineq} applies without assuming log
canonicity of the family being pulled back.

\subsection{Iteration}

Assume that every Frobenius stage $Z_n\to T_n$ is normal, with
reduced closed fiber $B_n$ and $\Q$-Cartier canonical divisor.
Let $E_n$ be the unique lift of a fixed vertical prime divisor $E_0$
over $Z_0=Z$. Set
\[
  A_n=A(E_n;Z_n,B_n),
  \qquad m_n=\ord_{E_n}(t_n)>0,
  \qquad e_n=e(E_{n+1}/E_n).
\]
Each successive function-field extension has degree $p$ by geometric
integrality of the generic fiber, and hence $e_n\in\{1,p\}$.
Since $t_n=t_{n+1}^p$,
\begin{equation}
  e_n m_n=p m_{n+1}.
  \label{eq:frob-multiplicity-recurrence}
\end{equation}
At each step, $m_n$ stays constant or is divided by $p$.

\begin{proposition}[Amplification of a negative log discrepancy]
\label{prop:amplification}
For every $n\ge0$,
\begin{equation}
  1\le m_n\le m_0,
  \qquad A_n\le\frac{p^n m_n}{m_0}A_0.
  \label{eq:frob-amplification}
\end{equation}
If $A_0<0$, then
\begin{equation}
  A_n\le\frac{p^n}{m_0}A_0\longrightarrow-\infty.
  \label{eq:frob-negative-divergence}
\end{equation}
\end{proposition}

\begin{proof}
The first assertion follows from
\eqref{eq:frob-multiplicity-recurrence}, since $m_n$ is a positive
integer. Iterating Proposition~\ref{prop:frob-ineq} gives
\[
  A_n\le\left(\prod_{i=0}^{n-1}e_i\right)A_0.
\]
Multiplying \eqref{eq:frob-multiplicity-recurrence} over
$i=0,\ldots,n-1$ gives
\begin{equation}
  \prod_{i=0}^{n-1}e_i=\frac{p^n m_n}{m_0}.
  \label{eq:frob-total-ramification}
\end{equation}
The empty product for $n=0$ is one, so
\eqref{eq:frob-amplification} holds at every stage. When $A_0<0$,
the inequality $m_n\ge1$ gives $m_nA_0\le A_0$ and hence
\eqref{eq:frob-negative-divergence}.
\end{proof}

\begin{remark}
The proof uses only positive integral fiber multiplicities; it
does not require $m_n$ to become prime to $p$. By
\eqref{eq:frob-total-ramification}, the total ramification index
grows at least as fast as $p^n/m_0$, even when the stabilized
multiplicity is divisible by $p$.
\cite[Proposition~6.1.3]{BenozzoPosva} establishes that failure of
local stability persists under further Frobenius base changes.
Proposition~\ref{prop:amplification} gives the explicit bound for
the lifts of a fixed divisor needed in the contradiction below.
\end{remark}

\section{Proof of the base-change criterion}\label{sec:proof}

\subsection{Simultaneous Frobenius twists of a separable cover}

Fix a finite separable cover $h:(C',c')\to(C,c)$ of degree $d$. Write
$C'_n=(C')^{1/p^n}$ and denote the induced cover by $h_n:C'_n\to C_n$.
The root-field construction gives the commutative diagram
\begin{equation}\label{eq:twisted-cover-diagram}
 \begin{array}{ccc}
 C'_n&\xrightarrow{\ h_n\ }&C_n\\
 \downarrow&&\downarrow\\
 C'&\xrightarrow{\ h\ }&C.
 \end{array}
\end{equation}
Associativity of fiber products gives
\begin{equation}\label{eq:simultaneous-base-change}
 X'_n:=(X\times_C C')\times_{C'}C'_n
 \simeq X_n\times_{C_n}C'_n.
\end{equation}
Both sides are canonically $X\times_C C'_n$, since the two
composites $C'_n\to C$ agree; the square need not be Cartesian.
The left side is the Frobenius tower of the covered family, and the
right side applies a separable cover to each Frobenius stage of the
original family. This isomorphism also identifies the scheme-theoretic
closed fibers.

\begin{lemma}\label{lem:constant-ramification}
The degree, local ramification index, and different exponent of $h_n$
are independent of $n$, as is the number $\beta$ in
\eqref{eq:base-different}.
\end{lemma}
\begin{proof}
The map $h_n$ is the coefficient Frobenius twist of $h$. The completed
local expansion
\[
 t=\sum_{j\ge b}c_j u^j
\]
becomes
\[
 t_n=\sum_{j\ge b}c_j^{1/p^n}u_n^j.
\]
The first nonzero exponent and the first nonzero exponent not divisible
by $p$ are unchanged, so $b$ and $d_h$ are preserved.
The field twist also preserves the function-field degree.
\end{proof}

\subsection{A uniform bound on the covered tower}

Suppose that $(X_n,D_n)$ is log canonical for every $n\ge0$.
Lemma~\ref{lem:reduced-fiber} gives normality of $X_n$, reducedness
of its closed fiber, and a $\Q$-Cartier canonical divisor.
By Proposition~\ref{prop:separable-normality} and
\eqref{eq:simultaneous-base-change}, the same holds for every
$X'_n$. The covered family therefore satisfies the hypotheses of
Section~\ref{sec:frobenius}, independently of whether its fiber pair
is log canonical.

Let $F_0$ be a vertical prime divisor over $X'=X'_0$, with unique
lifts $F_n$ through the purely inseparable tower. Put
\[
 \alpha_n=A(F_n;X'_n,D'_n),\qquad r_n=\ord_{F_n}(u_n).
\]
The multiplicity recurrence gives $1\le r_n\le r_0$. Restrict $F_n$ to
$k(X_n)$ and normalize the valuation to obtain a prime divisor
$E_n$ over $X_n$. Contraction along $X'_n\to X_n$ gives its center,
which is vertical. Proposition~\ref{prop:separable-comparison} and
Lemma~\ref{lem:constant-ramification} give
\begin{align}
 \alpha_n
 &=e(F_n/E_n)A(E_n;X_n,D_n)+\delta_{F_n}
       +1-e(F_n/E_n)-\beta r_n\notag\\
 &\ge1-d-\beta r_n\ge1-d-\beta r_0.
 \label{eq:uniform-lower-bound}
\end{align}
The bound is independent of $n$ for the lifts of each fixed divisor
$F_0$. Uniformity in $F_0$ is not needed, and effectiveness of the
different suffices without computing its exponent.

\subsection{The contradiction and the general case}

\begin{proof}[Proof of Theorem~\ref{thm:main}]
The implications (i)$\Rightarrow$(ii)$\Rightarrow$(iii) are immediate.
It remains to deduce (i) from (iii).

Suppose first that $h$ is separable. By the preceding subsection,
every stage of the covered Frobenius tower is normal, has reduced
closed fiber, and has $\Q$-Cartier canonical divisor. If $(X',D')$
were not log canonical, some prime divisor $F_0$ would satisfy
\[
 \alpha_0=A(F_0;X',D')<0.
\]
This divisor must be vertical: over the generic point of the base
the cover is \'etale, so log canonicity of $(X,D)$ excludes negative
log discrepancies for horizontal divisors after base change.

Proposition~\ref{prop:amplification} applied to $X'\to C'$ gives
\[
 \alpha_n\le\frac{p^n}{r_0}\alpha_0\longrightarrow-\infty,
\]
in contradiction with \eqref{eq:uniform-lower-bound}. Thus $(X',D')$
is log canonical for every finite separable cover.

For an arbitrary finite dominant cover,
Lemma~\ref{lem:factorization} gives a factorization
\[
 C'\xrightarrow{h_s}C_r\xrightarrow{q_r}C.
\]
The Frobenius stages of $X_r\to C_r$ are $X_{r+n}$, whose fiber pairs
are log canonical by (iii). Applying the separable case to $h_s$
therefore gives log canonicity of the pair on the actual fiber product
$X\times_C C'$, proving (i).
\end{proof}

Corollary~\ref{cor:wild} follows, including for non-Galois degree-$p$
covers.

\begin{remark}[A finite contradiction for each divisor]\label{rem:finite-contradiction}
For a divisor with $\alpha_0<0$, the contradiction occurs at any
stage $n$ for which
\[
 \frac{p^n(-\alpha_0)}{r_0}>d-1+\beta r_0.
\]
The required stage depends on the divisor. The argument gives no
uniform finite number of stages for all divisors of an arbitrary
family; in particular, the first Frobenius pullback cannot replace
all iterates on the strength of this argument.
\end{remark}

\begin{corollary}[Stability of the criterion under base change]\label{cor:base-change-stability}
If $X\to C$ is Frobenius stable and $C'\to C$ is a finite dominant
smooth-curve cover, then $X\times_C C'\to C'$ is Frobenius stable and
universally locally stable.
\end{corollary}
\begin{proof}
By Theorem~\ref{thm:main}, the family $X\to C$ is universally locally
stable. Composing any further finite smooth-curve cover with
$C'\to C$ proves the same for the pullback. Taking Frobenius covers
among these further covers gives Frobenius stability.
\end{proof}

\section{The algebraic perfection and a finite-stage formulation}\label{sec:perfection}

Let $R=\mathcal O_{\overline C,c}$. Write
\begin{equation}\label{eq:perfection}
 R^{\mathrm{perf}}=\varinjlim_n R^{1/p^n}
                  =\bigcup_{n\ge0}R^{1/p^n},
 \qquad C^{\mathrm{perf}}=\varprojlim_n C_n.
\end{equation}
Set $X_C^{\mathrm{perf}}=X\times_C C^{\mathrm{perf}}$. The subscript
records perfection of the base; the total space need not be perfect.

The limit schemes in \eqref{eq:perfection} need not be Noetherian.
The definition of log canonicity in \eqref{eq:discrepancies}, for normal
pairs essentially of finite type, does not apply directly to these
limits. We use the following finite-stage convention.

\begin{definition}\label{def:perfection}
The family is \emph{log canonical over the perfected base in the
finite-stage sense} if every pair $(X_n,D_n)$ is log canonical,
including the stage $n=0$.
\end{definition}

This definition concerns the algebraic Frobenius tower. It assigns
neither intrinsic discrepancies to $X_C^{\mathrm{perf}}$ nor a Cartier
boundary to its closed fiber.

\begin{corollary}[Perfection criterion]\label{cor:perfection}
A family satisfying the assumptions of Theorem~\ref{thm:main} is
universally locally stable if and only if it is log canonical over the
perfected base in the finite-stage sense.
\end{corollary}
\begin{proof}
Definition~\ref{def:perfection} is exactly condition (iii) of
Theorem~\ref{thm:main}.
\end{proof}

The condition is independent of the parameter: $C_n$ is defined by
$k(C)^{1/p^n}/k(C)$. The $p^n$-th root of any local parameter generates
this extension, since a one-variable function field over a perfect field
has $p$-rank one and a parameter is not a $p$-th power.

Corollary~\ref{cor:perfection} thus restates the finite base-change
criterion; it does not construct a theory of singularities on
non-Noetherian schemes. All perfections here are algebraic. The
construction involves no topology or analytic completion, and makes
no assertion about an analytic perfectoid space.

\section{A single first Frobenius base change does not suffice}\label{sec:first}

The example in Section~\ref{sec:arbitrary} shows that log canonicity of
the original fiber pair alone is insufficient. We now give, for every
prime $p$, a family for which the original pair and its first Frobenius
pullback are log canonical, but the second Frobenius pullback is not.
The relative dimension of this construction is $p$; thus it does not
give examples in a fixed dimension as $p$ varies. Its bad divisor is
horizontal, and the family is nonproper with non-log-canonical geometric
generic fiber. We also retain the characteristic-two surface example
at the end of the section.

As throughout the paper, log canonicity is required on the whole
localized total space, including at centers over the generic point of
the base. All pullbacks below are actual fiber products.

\subsection{A family in every positive characteristic}

Let $k$ be algebraically closed of characteristic $p>0$, and put
\[
 C=\Spec k[t]_{(t)},\qquad
 P(y)=\sum_{i=1}^{p}y_i^{p+1}.
\]
Consider
\begin{equation}\label{first:all:family}
 X=\Spec\frac{k[t]_{(t)}[x,x^{-1},y_1,\ldots,y_p]}
                   {(t-x^{p^2}-P(y))},
 \qquad f\colon X\longrightarrow C,
 \qquad D=(t=0).
\end{equation}
For $n\geq0$, set
\[
 C_n=\Spec k[t_n]_{(t_n)},\qquad t=t_n^{p^n},\qquad
 X_n=X\times_C C_n,\qquad D_n=(t_n=0).
\]
The maps are the identity on $k$, and $(X_0,D_0)=(X,D)$.

\begin{proposition}\label{first:all:proposition}
The morphism $f$ in \eqref{first:all:family} is flat, separated and
surjective of relative dimension $p$, with geometrically integral and
geometrically normal generic fiber. Its total space is regular and its
closed fiber is smooth. The pair $(X,D)$ is purely log terminal, and
$(X_1,D_1)$ is log canonical; in fact, $X_1$ is canonical.

For $n=0,1,2$, the actual fiber product $X_n$ is normal and Gorenstein,
and $D_n$ is a smooth effective Cartier divisor. There is a prime
divisor $E$ over $X_2$ such that
\begin{equation}\label{first:all:negative}
 \ord_E(D_2)=0,\qquad
 A(E;X_2,D_2)=1+p-p^2<0.
\end{equation}
Thus $(X_2,D_2)$ is not log canonical, for every prime $p$, including
$p=2$. The family is nonproper, and its geometric generic fiber is
not log canonical.
\end{proposition}

We first check the original family and the normality assertions.
Before localizing the base, eliminating $t$ identifies $X$ with
$\mathbb G_{m,x}\times\A^p_y$. Thus $X$ is regular. The map
$k[t]\to k[x,x^{-1},y_1,\ldots,y_p]$ given by
$t\mapsto x^{p^2}+P(y)$ is injective. After localization the target
is torsion-free over the base DVR, hence flat. The morphism is affine
and therefore separated. Both the generic and closed fibers are
nonempty hypersurfaces of dimension $p$, proving surjectivity and
the relative-dimension assertion.

The closed fiber is
\begin{equation}\label{first:all:closed}
 D=\{x^{p^2}+P(y)=0\},\qquad x\ne0.
\end{equation}
Its partial derivatives in the $y_i$ directions are $y_i^p$.
Their simultaneous vanishing on $D$ would force $x=0$, so $D$ is
smooth. It is a smooth Cartier divisor in a regular total space;
consequently $(X,D)$ is purely log terminal.

Write $s=t_1$ and $u=t_2$. The substitutions
$v=s-x^p$ and $r=u-x$ give, respectively,
\begin{equation}\label{first:all:models}
 S_1=\{v^p=P(y)\}\subset\A^{p+1}_{v,y},\qquad
 S_2=\{r^{p^2}=P(y)\}\subset\A^{p+1}_{r,y}.
\end{equation}
Thus $X_1$ and $X_2$ are localizations of
$S_1\times\mathbb G_{m,x}$ and $S_2\times\mathbb G_{m,x}$,
with base parameters $s=v+x^p$ and $u=r+x$.
Since $\partial P/\partial y_i=y_i^p\ne0$ in the rational function
field, $P$ is not a $p$-th power. Both equations in
\eqref{first:all:models} are therefore irreducible. Their only
singular point is the origin. Each has dimension $p\geq2$, so
Cohen--Macaulayness and regularity in codimension one prove normality.
Both products and their localizations are normal
and Gorenstein. For $n=0,1,2$, the fiber $D_n$ has equation
\eqref{first:all:closed}; flatness makes its defining parameter
a nonzerodivisor. Hence it is a smooth effective Cartier divisor.

For the geometric generic fiber, take $K=\overline{k(t)}$ and
$\alpha\in K$ with $\alpha^{p^2}=t$. Setting $r=\alpha-x$ identifies
it with
\begin{equation}\label{first:all:generic}
 X_{\bar\eta}\cong (S_2)_K\setminus\{r=\alpha\}.
\end{equation}
The same irreducibility and normality arguments apply over $K$.
Moreover, this open subset contains the origin because $\alpha\ne0$.

\subsection{The first Frobenius pullback is canonical}

We prove that $S_1$ is canonical by an explicit blowup computation.
First recall a calculation valid in every characteristic. The normal
surface
\[
 H_m=\{UV=W^m\}\subset\A^3,\qquad m\geq2,
\]
has a crepant resolution obtained by repeatedly blowing up its
singular point. Indeed, after the first ambient blowup the $U$-chart
is $V_1=U^{m-2}W_1^m$, which is smooth; the $V$-chart is similar.
The $W$-chart is
\[
 U_1V_1=W^{m-2}.
\]
This chart is smooth for $m=2,3$, and otherwise has the same form
with smaller exponent. At each step the ambient exceptional
discrepancy and the hypersurface multiplicity are both two, so
adjunction makes the strict-transform morphism crepant. The strict
transforms are normal, and the process ends in a smooth surface.
This proves that $H_m$ is canonical. Taking products with smooth
affine spaces and making \'etale base changes preserves these
explicit crepant resolutions.

Now blow up the origin in $\A^{p+1}_{v,y}$, and let
$\pi\colon\widetilde S_1\to S_1$ be the strict transform.
In the $y_j$-chart, write
\[
 y_j=z,\qquad v=za,\qquad y_i=zw_i\quad(i\ne j).
\]
The strict-transform equation is
\begin{equation}\label{first:all:first-chart}
 a^p=zF_j(w),\qquad F_j(w)=1+\sum_{i\ne j}w_i^{p+1}.
\end{equation}
Where $F_j\ne0$, one can eliminate $z$, so this chart is smooth.
On $F_j=0$, at least one $w_i$ is nonzero, and hence $dF_j\ne0$.
Using $F_j$ as one of a system of \'etale local coordinates shows
that \eqref{first:all:first-chart} is \'etale locally
\[
 \{a^p=zb\}\times\A^{p-2}.
\]
It is therefore normal and canonical by the preceding calculation.
The remaining ambient chart is the $v$-chart, with equation
\[
 1=v\sum_{i=1}^p w_i^{p+1}.
\]
Here $v$ is invertible, so this chart lies over the original smooth
locus and does not meet the exceptional fiber. This checks every
chart of $\widetilde S_1$.

The multiplicity of $S_1$ at the origin is $p$. The blowup of its
smooth ambient space of dimension $p+1$ has exceptional discrepancy
$p$. If $E_{\mathrm{amb}}$ denotes the ambient exceptional divisor,
hypersurface adjunction gives
\begin{equation}\label{first:all:crepant}
 K_{\widetilde S_1}
 =\pi^*K_{S_1}+(p-p)E_{\mathrm{amb}}|_{\widetilde S_1}
 =\pi^*K_{S_1}.
\end{equation}
The cancellation is valid even though the restricted exceptional
Cartier divisor need not be reduced. Since $\widetilde S_1$ is
canonical and $\pi$ is crepant, $S_1$ is canonical as well.
The same explicit local resolutions, after taking their products
with $\mathbb G_m$, show that $X_1$ is canonical.

The singular locus of $S_1\times\mathbb G_{m,x}$ is given by
$v=y_1=\cdots=y_p=0$. There $s=v+x^p=x^p\ne0$, so it is disjoint
from $D_1$. Near $D_1$ the total space and the boundary are smooth,
while away from $D_1$ the total space is canonical. Thus
$(X_1,D_1)$ is log canonical.

\subsection{A negative discrepancy after the second Frobenius pullback}

We compute on $S_2$ using a smooth birational chart. Put
\[
 q=p^2,\qquad d=p+1,\qquad h=p-1,
 \qquad q=hd+1,
\]
and define
\[
 F=1+\sum_{i=1}^{p-1}c_i^{p+1},\qquad
 T=\Spec k[\lambda,c_1,\ldots,c_{p-1},F^{-1}].
\]
The formulas
\begin{equation}\label{first:all:second-chart}
 r=\lambda^dF,\qquad
 y_i=\lambda^qF^h c_i\ (i<p),\qquad
 y_p=\lambda^qF^h
\end{equation}
define a morphism $\rho\colon T\to S_2$: both sides of
$r^q=\sum y_i^d$ become $\lambda^{dq}F^q$.
Over $ry_p\ne0$ the inverse is
\begin{equation}\label{first:all:inverse}
 \lambda=\frac{y_p}{r^h},\qquad c_i=\frac{y_i}{y_p}.
\end{equation}
Indeed, $F=r^q/y_p^d$ there, and $q=hd+1$ recovers every formula
in \eqref{first:all:second-chart}. Thus $\rho$ is birational.
The divisor $E_S=(\lambda=0)$ is integral and reduced, $T$ is
smooth, and $\lambda$ is a uniformizer at its generic point.
Taking the closure of the graph in $S_2\times\mathbb P^p$ and
normalizing gives a proper normal birational model containing
this smooth chart. In particular $E_S$ is a prime divisor over
$S_2$, centered at its origin. Its orders are
\[
 \ord_{E_S}(r)=p+1,\qquad
 \ord_{E_S}(y_i)=p^2\quad(1\leq i\leq p).
\]

Up to a nonzero constant, the hypersurface residue generator of
the dualizing sheaf of $S_2$ is
\[
 \omega=\frac{dr\wedge dy_1\wedge\cdots\wedge dy_{p-1}}{y_p^p}.
\]
We compute its pullback directly, so no assertion about a weighted
exceptional multiplicity is needed. In characteristic $p$,
\[
 dr=\lambda^pF\,d\lambda+\lambda^{p+1}dF,\qquad
 dy_i=\lambda^qF^h
       \left(dc_i+h c_i\frac{dF}{F}\right)\quad(i<p).
\]
Writing $c$ for the column vector $(c_1,\ldots,c_{p-1})^T$, the
coefficient matrix in the second formula has determinant
\[
 \det\left(I+\frac{h}{F}c(\nabla F)^T\right)
 =1+\frac{h}{F}\sum_{i=1}^{p-1}c_i\frac{\partial F}{\partial c_i}
 =1+\frac{h(F-1)}F
 =\frac1F.
\]
Here $h=-1$ in $k$ and
$\sum c_i(\partial F/\partial c_i)=F-1$. The $dF$ term in $dr$
wedges to zero against the other $p-1$ differentials. Since
$y_p^p=\lambda^{qp}F^{hp}$, we obtain
\begin{equation}\label{first:all:residue}
 \rho^*\omega
 =\lambda^{p-p^2}F^{1-p}\,
   d\lambda\wedge dc_1\wedge\cdots\wedge dc_{p-1},
\end{equation}
up to a nonzero constant. The factor $F$ is a unit. Therefore
\begin{equation}\label{first:all:discrepancy}
 a(E_S;S_2,0)=p-p^2,\qquad
 A(E_S;S_2,0)=1+p-p^2<0.
\end{equation}

Take the product of the proper birational model with
$\mathbb G_{m,x}$ and denote the resulting prime divisor by $E$.
Wedging \eqref{first:all:residue} with $dx$ leaves its discrepancy
unchanged. The base parameter satisfies
\[
 u=r+x\equiv x\pmod E,
\]
so $\ord_E(u)=0$. This divisor survives the localization defining
$X_2$: every inverted polynomial $g(u)$ with $g(0)\ne0$ has
nonzero residue $g(x)$, since $x$ is transcendental over $k$
in the residue field of $E$. Localizing the proper birational
model gives a proper birational model over the actual fiber
product $X_2$. Subtracting the boundary order now gives
\[
 A(E;X_2,D_2)=A(E;X_2,0)-\ord_E(u)=1+p-p^2<0,
\]
as asserted in \eqref{first:all:negative}.

\subsection{Dimension and scope of the construction}

The center just constructed is horizontal and avoids $D_2$.
The original family is not even a closed map: the nonempty closed
subscheme
\[
 Z=V(y_1,\ldots,y_p)\subset X
\]
satisfies $t=x^{p^2}$, and its image is exactly the generic point
of $C$. Hence $f$ is not proper. By \eqref{first:all:generic},
the geometric generic fiber contains the origin of $(S_2)_K$.
The same calculation over $K$ gives the negative log discrepancy
\eqref{first:all:discrepancy}, so that fiber is not log canonical.
This completes the proof of Proposition~\ref{first:all:proposition}.

On the other hand, $U=X\setminus Z$ contains all of $D$ and is
smooth over $C$, because at least one of the relative partial
derivatives $y_i^p$ is nonzero. After every regular-curve base
change, $U$ remains smooth over the new base and its fiber pair
is purely log terminal. Thus the example does not contradict
inversion of adjunction along the closed fiber, nor does it
answer the proper or geometrically log canonical generic-fiber
questions of Section~\ref{sec:proper}.

For $p=2$ this is already a surface family, with
$A(E;X_2,D_2)=-1$. For general $p$ it has relative dimension $p$
and total dimension $p+1$. It proves that a single first Frobenius
test is insufficient in every positive characteristic, but does
not assert the same conclusion in a fixed relative dimension
as the characteristic grows.

\subsection{A further surface example in characteristic two}

There is also a characteristic-two surface example for which the
first pullback has an $E_8^0$ singularity and the second pullback
has a divisor of log discrepancy $-2$. We give its full calculation
below.

\subsubsection{The family}

Let $k$ be an algebraically closed field of characteristic two and put
\[
 C=\Spec k[t]_{(t)}.
\]
Consider the family
\begin{equation}\label{first:eq:family}
 f\colon X=
 \Spec\frac{k[t]_{(t)}[x,x^{-1},z,w]}
                 {(t-x^4-z^3-w^5)}
 \longrightarrow C,
 \qquad D=(t=0)\subset X.
\end{equation}
For $n\geq0$, let
\[
 C_n=\Spec k[t_n]_{(t_n)},\qquad
 C_n\longrightarrow C,\quad t\longmapsto t_n^{2^n},
\]
where the map is the identity on $k$. Set
\[
 X_n=X\times_C C_n,\qquad D_n=(t_n=0)\subset X_n.
\]
In particular, $(X_0,D_0)=(X,D)$. These are actual fiber products;
no normalization is taken.

\begin{proposition}\label{first:prop:example}
The morphism $f$ is flat, separated, and surjective, with geometrically
normal and geometrically integral generic fiber. For $n=0,1,2$, the
scheme $X_n$ is normal and Gorenstein, and $D_n$ is a smooth effective
Cartier divisor. The pairs $(X,D)$ and $(X_1,D_1)$ are log canonical,
whereas $(X_2,D_2)$ is not.
More precisely, there is a prime divisor $E$ over $X_2$ such that
\[
 \ord_E(D_2)=0,\qquad A(E;X_2,D_2)=-2.
\]
\end{proposition}

We first check the original family and its generic fiber.
Before localizing the base, eliminating $t$ in
\eqref{first:eq:family} identifies the total space with
$\mathbb G_{m,x}\times\A^2_{z,w}$. Thus $X$ is regular. Its coordinate ring is a
domain, and the map $k[t]_{(t)}\to\mathcal O(X)$ is injective. It is
therefore torsion-free, hence flat, over this discrete valuation ring.
The morphism is affine, hence separated. Both its generic and closed
fibers are nonempty, so it is surjective.

The closed fiber has equation
\begin{equation}\label{first:eq:closed-fiber}
 x^4+z^3+w^5=0,\qquad x\ne0.
\end{equation}
Its partial derivatives with respect to $z$ and $w$ are $z^2$ and $w^4$.
If both vanished, the equation would force $x=0$. Hence $D$ is smooth.
Consequently $(X,D)$ is a simple normal crossing pair and is log
canonical (indeed, purely log terminal).

To check the geometric generic fiber, let
$K=\overline{k(t)}$ and choose $\alpha\in K$ with $\alpha^4=t$.
After the change of variable $r=x+\alpha$, its equation is
\begin{equation}\label{first:eq:geometric-generic}
 r^4+z^3+w^5=0,\qquad r+\alpha\ne0.
\end{equation}
The element $z^3+w^5$ is not a square in $K(z,w)$: its derivative with
respect to $z$ is nonzero. Thus $r^4+z^3+w^5$ is irreducible over
$K(z,w)$, and the hypersurface is integral. Its only singular point is
the origin. A hypersurface is Cohen--Macaulay, so this integral surface
is normal by Serre's criterion. The open set in
\eqref{first:eq:geometric-generic} is nonempty and contains the origin, since
$\alpha\ne0$. This proves the assertions about the generic fiber.

\subsubsection{The first Frobenius base change}

Write $s=t_1$. The first pullback is
\[
 X_1=
 \Spec\frac{k[s]_{(s)}[x,x^{-1},z,w]}
                 {(s^2-x^4-z^3-w^5)}.
\]
Putting $v=s+x^2$ gives
\begin{equation}\label{first:eq:first}
 v^2+z^3+w^5=0.
\end{equation}
Thus $X_1$ is a localization of $S_1\times\mathbb G_{m,x}$, where
\[
 S_1=\Spec k[v,z,w]/(v^2+z^3+w^5),
\]
and the base parameter is $s=v+x^2$.

The origin of $S_1$ is the rational double point of type $E_8^0$ in
characteristic two; see \cite[\S1.3 and Table~1]{LMM}.
In particular, $S_1$ is canonical. Its minimal resolution is crepant,
and taking the product of this resolution with $\mathbb G_{m,x}$ shows that
$S_1\times\mathbb G_{m,x}$ is canonical as well. This property persists under
localization, so $X_1$ is canonical.

The singular locus of $S_1\times\mathbb G_{m,x}$ is given by $v=z=w=0$.
On this locus,
\[
 s=v+x^2=x^2\ne0.
\]
It is therefore disjoint from $D_1=(s=0)$. The closed fiber $D_1$ has
the same equation as \eqref{first:eq:closed-fiber}, so it is smooth and lies
in the regular locus of $X_1$. Near $D_1$ the pair is simple normal
crossing, while on $X_1\setminus D_1$ the total space is canonical.
Hence $(X_1,D_1)$ is log canonical.

\subsubsection{The second Frobenius base change}

Write $u=t_2$, so that $t=u^4$. Setting $r=u+x$ gives
\begin{equation}\label{first:eq:second}
 X_2=\Spec\frac{k[u]_{(u)}[x,x^{-1},z,w]}
                         {(u^4-x^4-z^3-w^5)},
 \qquad r^4+z^3+w^5=0.
\end{equation}
Equivalently, $X_2$ is a localization of $S\times\mathbb G_{m,x}$, where
\[
 S=\Spec k[r,z,w]/(r^4+z^3+w^5),\qquad u=r+x.
\]
The same irreducibility and Jacobian calculations used above show that
$S$ is an integral normal surface with its only singular point at the
origin. It follows that $X_2$ is normal. All three total spaces
$X_0,X_1,X_2$ are hypersurfaces in regular schemes, hence Gorenstein.
Each $D_n$, for $n=0,1,2$, is the smooth fiber
\eqref{first:eq:closed-fiber} cut out by the nonzerodivisor $t_n$.
In particular, all the log canonical divisors under discussion are
Cartier.

\paragraph{A divisor of negative log discrepancy}

We compute a discrepancy on $S$ directly. Consider
\[
 T=\Spec\frac{k[\lambda,a,b,b^{-1}]}
                   {(a^4+b^3+\lambda^3b^5)}
\]
and the morphism $\pi\colon T\to S$ defined by
\begin{equation}\label{first:eq:chart}
 r=\lambda^3a,\qquad z=\lambda^4b,\qquad w=\lambda^3b.
\end{equation}
Indeed, substitution in the equation of $S$ gives
$\lambda^{12}(a^4+b^3+\lambda^3b^5)$.
On the open set $zw\ne0$, the inverse is
\[
 \lambda=\frac zw,\qquad
 a=\frac{rw^3}{z^3},\qquad
 b=\frac{w^4}{z^3}.
\]
Thus $\pi$ is birational. The curve
\[
 E_S=(\lambda=0)\subset T,\qquad a^4+b^3=0,
\]
is integral and reduced; the binomial is irreducible since $3$ and $4$
are coprime. Along $E_S$,
\[
 \frac{\partial}{\partial b}
       (a^4+b^3+\lambda^3b^5)
 =b^2(1+\lambda^3b^2)
\]
is a unit. Hence $T$ is smooth near $E_S$, and $\lambda$ is a
uniformizer at its generic point. This gives a divisorial valuation
over $S$; a proper birational model containing this generic point is
obtained by taking the closure of the graph in
$S\times\mathbb P^3$ and then normalizing. Its center on $S$ is the
origin. Moreover,
\begin{equation}\label{first:eq:weights}
 \ord_{E_S}(r)=3,\qquad
 \ord_{E_S}(z)=4,\qquad
 \ord_{E_S}(w)=3.
\end{equation}

A rational expression for a generator of the dualizing sheaf of $S$ is
\[
 \omega=\frac{dr\wedge dw}{z^2}.
\]
This is the hypersurface residue form, using
$\partial(r^4+z^3+w^5)/\partial z=z^2$.
Put $Q=1+\lambda^3b^2$. Differentiating the equation of $T$ in
characteristic two gives
\[
 b^2Q\,db+\lambda^2b^5\,d\lambda=0,
 \qquad
 db=\frac{\lambda^2b^3}{Q}\,d\lambda.
\]
Using \eqref{first:eq:chart}, we obtain
\[
 dw=\frac{\lambda^2b}{Q}\,d\lambda,
 \qquad
 dr=\lambda^2a\,d\lambda+\lambda^3\,da,
\]
and hence
\begin{equation}\label{first:eq:residue}
 \pi^*\omega
 =\frac{\lambda^{-3}}{bQ}\,da\wedge d\lambda.
\end{equation}
Both $b$ and $Q$ are units at the generic point of $E_S$, and
$da\wedge d\lambda$ generates the canonical sheaf there. It follows
that
\begin{equation}\label{first:eq:discrepancy}
 a(E_S;S,0)=-3,\qquad A(E_S;S,0)=-2.
\end{equation}
This also agrees with the weight calculation
$3+4+3-12=-2$ from \eqref{first:eq:weights}.

\paragraph{The fiber boundary and the localization}

Take the product of the preceding birational model with $\mathbb G_{m,x}$ and
denote the resulting divisor by $E$. The discrepancy is unchanged;
one may equally compute with $\omega\wedge dx$ in
\eqref{first:eq:residue}. On $E$ we have
\[
 u=r+x\equiv x.
\]
In particular, $\ord_E(u)=0$. The valuation also survives the
localization defining $X_2$: every polynomial $g(u)\in k[u]$ with
$g(0)\ne0$ has residue $g(x)\ne0$ at the generic point of $E$, since
$x$ is transcendental over $k$ there. Thus every element inverted in
$k[u]_{(u)}$ has valuation zero, and $E$ has a center on $X_2$.

Since $D_2=(u=0)$, equation~\eqref{first:eq:discrepancy} now gives
\[
 A(E;X_2,D_2)
 =A(E;X_2,0)-\ord_E(u)
 =-2.
\]
Therefore $(X_2,D_2)$ is not log canonical. This completes the proof
of Proposition~\ref{first:prop:example}.

\subsubsection{What the example does and does not show}

The divisor $E$ is horizontal: its center maps to the generic point
of $C_2$ and does not meet $D_2$. This feature is essential.
The original morphism $f$ is not even closed. Indeed, the nonempty
closed subscheme $Z=(z=w=0)\subset X$ satisfies $t=x^4$, so $t$ is
invertible on $Z$. Its image is exactly the generic point of $C$.
In particular, $f$ is not proper.

The geometric generic fiber is not log canonical either. It contains
the origin of the surface \eqref{first:eq:geometric-generic}, and the same
calculation \eqref{first:eq:chart}--\eqref{first:eq:discrepancy}, now over
$\overline{k(t)}$, gives log discrepancy $-2$ there.

In fact, deleting $Z$ from $X$ leaves a smooth morphism to $C$:
outside $Z$, at least one of the partial derivatives $z^2,w^4$ is
nonzero. This open subset contains all of $D$. Every base change of
this restricted family is smooth. Consequently the example gives
no obstruction to a statement made only after shrinking $X$ around
its closed fiber.

For the whole family \eqref{first:eq:family}, however, it proves that
log canonicity of the original fiber pair and of its first Frobenius
pullback does not imply log canonicity after every finite Frobenius
base change. No conclusion about the corresponding question for
proper families, or for families with geometrically log canonical
generic fiber, follows from this example.

\section{Two further valuation examples}\label{sec:examples}

The first example shows that wild base change can decrease a log
discrepancy even for a family whose base changes are all smooth. The
second shows that a vertical multiplicity can remain divisible by the
characteristic throughout the Frobenius tower. The proof of
Theorem~\ref{thm:main} combines the amplification of negative discrepancies
with a uniform lower bound; it requires neither preservation of individual
discrepancies nor eventual tameness.

\subsection{Discrepancies need not be preserved}

\begin{example}\label{ex:discrepancies}
Let $\operatorname{char} k=2$.  On $\mathbb A^2_{a,b}$ consider the smooth
family
\[
  t=a^6+b+a^9.
\]
In the coordinates $(a,t)$ this is the trivial affine-line family over
$\mathbb A^1_t$, an open chart of the smooth proper family
$\mathbb P^1_a\times\mathbb A^1_t$. Every smooth-curve base change is
smooth and has smooth total space.

Take the finite separable degree-two cover
\begin{equation}\label{eq:example-wild-cover}
  u^2+tu+t=0,\qquad t=\frac{u^2}{1+u}.
\end{equation}
In its coordinate ring, $1+u$ is a unit because
$(1+u)(1+u+t)=1$. At $u=t=0$ the covering curve is smooth and the cover
is totally ramified. The expansion $t=u^2+u^3+\cdots$ gives the
ramification index, different exponent, and logarithmic ramification
coefficient
\[
  b=2,\qquad d_h=2,\qquad \beta=1.
\]

Let $E$ be the monomial divisorial valuation with
$\ord_E(a)=1$ and $\ord_E(b)=8$.  A birational chart carrying $E$ is
\[
  a=x,\qquad b=x^8y,\qquad E=(x=0),
\]
and on this chart
\[
  t=x^6(1+x^2y+x^3).
\]
This is the chart of the blowup of $(a^8,b)$ on which $a^8$ generates
the pulled-back ideal. After the cover \eqref{eq:example-wild-cover},
substitute $u=x^3+x^4z$ and divide the covering equation by $x^8$ to obtain
\begin{equation}\label{eq:example-normalization}
  (1+x^3+x^4z)y=z^2+x^2z+x^4+x^5z.
\end{equation}

To check finiteness and normality, let
\[
  A=k[x,y,u]\big/
  \big(u^2+(x^6+x^8y+x^9)u+x^6+x^8y+x^9\big)
\]
be the ring of the actual pulled-back chart. In its fraction field,
$z=(u+x^3)/x^4$ satisfies the monic equation
\[
  z^2+(x^2+x^4y+x^5)z+(1+x^3)y+x^4=0.
\]
It follows that $A[z]$ is finite over $A$ and, by the expression for
$z$, birational. In $A[z]$, the coefficient $1+x^3+x^4z=1+u$ is a unit.
We can therefore eliminate $y$ from \eqref{eq:example-normalization}
and identify the ring with
\[
  k[x,z,1/(1+x^3+x^4z)].
\]
This ring is smooth and normal, hence is the normalization of $A$.
Let $F=(x=0)$ on the normalized chart.

Downstairs we have
\[
  da\wedge db=x^8\,dx\wedge dy,\qquad \ord_E(t)=6.
\]
Consequently
\begin{equation}\label{eq:example-discrepancy-downstairs}
  A(E;X,D)=1+8-6=3.
\end{equation}
The pulled-back total space has smooth coordinates $(a,u)$, and
\[
  da\wedge du=x^4\,dx\wedge dz,\qquad \ord_F(u)=3,
\]
so
\begin{equation}\label{eq:example-discrepancy-upstairs}
  A(F;X',D')=1+4-3=2.
\end{equation}
For comparison with Proposition~\ref{prop:separable-comparison}, compute
\[
  \frac{\partial y}{\partial z}
  =\frac{x^2(1+x^2z^2)}{(1+x^3+x^4z)^2}.
\]
Thus $e(F/E)=1$ and $\delta_F=2$. The different exponent can also be
read from the derivative with respect to $z$ of the monic equation,
which has order two along $F$. The discrepancy formula gives
\[
  A(F;X',D')=3+2-3=2.
\]

For the Frobenius change $t=v^2$, put $v=x^3+x^4w$. The normalized
equation is $y=w^2+x$: the element $w$ is integral over the pulled-back
chart, and the finite birational ring $k[x,w]$ is regular. The same
top-form calculation gives log discrepancy $2$. Both base changes
therefore decrease the original discrepancy.
\end{example}

\subsection{Multiplicities may remain divisible by the characteristic}

\begin{example}\label{ex:persistent}
Let $X=\mathbb P^1_x\times\mathbb A^1_t\to\mathbb A^1_t$ be the smooth
proper product family in characteristic $p>0$. On the affine chart
$(x,t)$, blow up the ideal
\[
  (x^{2p+1},t-x^p-x^{p+1}).
\]
Extend this ideal by the unit ideal off $x=0$. The chart on which
$x^{2p+1}$ generates the pulled-back ideal is $\mathbb A^2_{x,y}$, with
\begin{equation}\label{eq:persistent-chart}
  t=x^p+x^{p+1}+yx^{2p+1}.
\end{equation}
The divisor $E=(x=0)$ has $\ord_E(t)=p$.

The divisorial DVR and its completion are
\[
  A=k[x,y]_{(x)},\qquad \widehat A=k(y)[[x]].
\]
In $\widehat A[T]$, consider
\[
  H(T)=T^p+T^{p+1}-t.
\]
At $T=x$ we have
\[
  H(x)=-yx^{2p+1},\qquad H'(x)=x^p,
\]
and therefore
\[
  \ord_x(H(x))=2p+1>2p=2\ord_x(H'(x)).
\]
The strong Hensel criterion gives a root $\pi=x+O(x^{p+1})$ in
$\widehat A$. Since $\pi$ has linear coefficient one, it is a
uniformizer and defines a formal change of parameter over $k(y)$.
Hence
\begin{equation}\label{eq:persistent-hensel}
  k(y)[[x]]=k(y)[[\pi]],\qquad t=\pi^p+\pi^{p+1}.
\end{equation}

Fix $n\geq 0$, put $q=p^n$, and write
\[
  F=\operatorname{Frac}(\widehat A)=k(y)((\pi)),\qquad
  \pi_n=\pi^{1/q},\qquad t_n=t^{1/q}.
\]
The element $t$ is not a $p$-th power in $F$, since its formal derivative
with respect to $\pi$ is $\pi^p\ne0$. Therefore
\[
  [F(t_n):F]=q.
\]
The extension $k(y)((\pi_n))/F$ also has degree $q$, and it contains
\[
  t_n=\pi_n^p+\pi_n^{p+1}.
\]
Thus
\begin{equation}\label{eq:persistent-complete-fields}
  F(t_n)=k(y)((\pi_n)).
\end{equation}

Now let $L=\operatorname{Frac}(A)=k(X)$, $L_n=L(t_n)$, and let $A_n$
be the normalization of $A$ in $L_n$. Excellence of $A$ makes this
normalization finite. Since $L_n/L$ is purely inseparable, $A_n$ is
the unique DVR above $A$, corresponding to the lifted divisor $E_n$.
The polynomial defining $L_n/L$ remains irreducible over $F$ by
\eqref{eq:persistent-complete-fields}, so $L_n\otimes_L F=F(t_n)$ is a
field. Finite normalization commutes with completion for excellent DVRs;
thus $\widehat{A_n}$ is the integral closure of $\widehat A$ in $F(t_n)$.
By \eqref{eq:persistent-complete-fields},
\[
  \widehat{A_n}=k(y)[[\pi_n]].
\]
Since completion preserves the normalized valuation of a DVR,
\begin{equation}\label{eq:persistent-multiplicity}
  \ord_{E_n}(t_n)=p\qquad\text{for every }n\geq0.
\end{equation}
The multiplicity therefore remains divisible by $p$ at every stage.

The family and all its smooth-curve base changes are smooth. On the
chart \eqref{eq:persistent-chart},
\[
  dx\wedge dt=x^{2p+1}\,dx\wedge dy,\qquad
  A(E;X,X_0)=1+(2p+1)-p=p+2>0.
\]
Proposition~\ref{prop:amplification} uses positivity of the multiplicities
to amplify negative discrepancies; it does not require the multiplicities
to become prime to $p$.
\end{example}

\section{Proper families and geometrically log canonical generic fibers}\label{sec:proper}

The counterexamples in Sections~\ref{sec:arbitrary} and~\ref{sec:first}
concern families whose geometric generic fibers are not log canonical.
They leave open the corresponding question for proper families with
geometrically log canonical generic fiber. We first treat surface
families in characteristic greater than five, then give an exclusion
for locally complete intersection total spaces in every dimension and
characteristic. In relative dimension three, we analyze an additive
quotient construction and formulate a discrepancy criterion for a
possible counterexample. These arguments establish neither an example
satisfying all the strengthened conditions nor a general nonexistence
theorem in relative dimension three.

\subsection{The question and conventions}

\begin{question}\label{proper:q:proper}
Can a proper flat family over a smooth curve germ have a log canonical
fiber pair and a geometrically log canonical generic fiber, while
some Frobenius base change has a non-log-canonical fiber pair?
\end{question}

Let $k$ be an algebraically closed field of characteristic $p>0$, and set
\[
 R=k[t]_{(t)},\qquad C=\Spec R.
\]
This is the local scheme of the affine line at the origin. Let
$f\colon X\to C$ be flat and proper, with $X$ normal and integral, and
write $D=X_0$ for the \emph{scheme-theoretic} closed fiber. Unless
explicitly stated otherwise, assume
\begin{equation}\label{proper:eq:assumptions}
 (X,D)\text{ is log canonical},
 \qquad
 X_{\bar\eta}\text{ is log canonical}.
\end{equation}
There is no additional horizontal boundary in this section. The phrase
\emph{the original pair} always means $(X,D)$, not merely $(X,0)$.
In particular, $K_X+D$ is $\Q$-Cartier.

For $q=p^e$, let
\[
 C^{[q]}=\Spec k[s]_{(s)}\longrightarrow C,\qquad t\longmapsto s^q,
 \qquad
 X^{[q]}=X\times_C C^{[q]},\qquad D^{[q]}=(X^{[q]})_0.
\]
The bracketed superscript records the degree $q$ of the base change,
rather than its iteration index. We use this $k$-linear convention
for Frobenius. Thus
$D^{[q]}\cong D$ as $k$-schemes. All base changes are actual fiber products;
normalization is not built into the notation.

For a divisor $E$ over a normal pair $(V,B)$, our conventions are
\[
 a(E;V,B)=\operatorname{coeff}_E(K_{\widetilde V}-\pi^*(K_V+B)),
 \qquad A(E;V,B)=a(E;V,B)+1,
\]
where $\pi\colon\widetilde V\to V$ is a normal birational model on
which $E$ appears. Log canonicity means $a\geq-1$, equivalently
$A\geq0$, for every such divisor.

\subsection{Preliminary reductions}

\begin{lemma}\label{proper:lem:normal}
Under \eqref{proper:eq:assumptions}, $D$ is reduced, and each $X^{[q]}$ is normal.
Moreover, $K_{X^{[q]}}+D^{[q]}$ is $\Q$-Cartier.
\end{lemma}

\begin{proof}
Since $X$ is normal, it satisfies $S_2$. The Cartier divisor $D$
therefore satisfies $S_1$. Log canonicity forces each multiplicity of
$D$ to be one, so $D$ is generically reduced. Together these statements
show that $D$ is reduced.

The morphism $g\colon X^{[q]}\to X$ is finite flat and is a universal
homeomorphism. Flat local extensions with zero-dimensional fibers
preserve $S_2$, so $X^{[q]}$ is $S_2$. Its generic fiber is normal by
geometric normality of $X_\eta$. At a generic point of a component of
$D^{[q]}$, the local ring modulo $s$ is a field, since $D^{[q]}\cong D$ is
reduced. This one-dimensional local ring consequently has maximal
ideal generated by $s$ and is regular. Thus $X^{[q]}$ is regular in
codimension one. Serre's criterion proves normality.

For the $\Q$-Cartier assertion, work locally on $X$ and choose $r>0$
such that $rK_X$ is Cartier; this is possible because $D$ is Cartier.
On the regular locus $U\subset X$, the cover is the monic
hypersurface
\[
 \cO_U[s]/(s^q-t).
\]
Hypersurface adjunction identifies its canonical sheaf with
$g^*\omega_U$. This is an adjunction isomorphism, not ordinary
pullback of differential forms along an inseparable morphism.
The complement of $g^{-1}(U)$ has codimension at least two.
Consequently the line bundle $g^*\cO_X(rK_X)$ agrees in codimension
one with $\omega_{X^{[q]}}^{[r]}$, and the two are isomorphic by reflexive
extension. Hence $K_{X^{[q]}}$, and therefore $K_{X^{[q]}}+D^{[q]}$, is
$\Q$-Cartier.
\end{proof}

\begin{lemma}\label{proper:lem:proper}
Let $V\to\Spec R$ be proper, where $R$ is a local DVR.
Every open subset of $V$ containing the entire closed fiber equals $V$.
\end{lemma}

\begin{proof}
Its complement is closed, so its image in $\Spec R$ is closed by
properness. If the complement were nonempty, its image would be a
nonempty closed subset avoiding the closed point. No such subset of
the spectrum of a local DVR exists.
\end{proof}

There are two useful ways to finish a local argument. If log
canonicity has been established on an open neighborhood of the
whole closed fiber, Lemma~\ref{proper:lem:proper} gives the global conclusion.
Alternatively, one can check all vertical divisorial valuations
directly and use geometric log canonicity of the generic fiber to
control horizontal valuations.

The examples in Sections~\ref{sec:arbitrary} and~\ref{sec:first}
do not satisfy the strengthened hypotheses: their geometric generic
fibers already contain non-log-canonical singularities. A
compactification retaining any of these families as an open subscheme retains
these singularities on the geometric generic fiber. Properness also
rules out their particular mechanism of failure: a nonempty closed
bad locus cannot lie entirely over the generic point while avoiding
the closed fiber.

\subsection{Families of surfaces in characteristic greater than five}

For surface fibers, depth and inversion of adjunction give a negative
answer to Question~\ref{proper:q:proper} in characteristic greater than five.

\begin{proposition}\label{proper:prop:surfaces}
Assume \eqref{proper:eq:assumptions}, $\dim X=3$, and $p>5$.
Then
\[
 (X^{[q]},D^{[q]})\textnormal{ is log canonical for every }q=p^e.
\]
In particular, a proper family of surfaces satisfying
\eqref{proper:eq:assumptions} cannot give a counterexample to
Question~\ref{proper:q:proper} in characteristic greater than five.
\end{proposition}

\begin{proof}
First one needs more than reducedness of the closed fiber.
The depth theorem for three-dimensional log canonical pairs with
effective Cartier boundary gives $S_3$ along $D$
\cite[Theorem~4.10]{ABP}. Since $D$ is Cartier, it follows that $D$
is $S_2$. The codimension-one adjunction analysis then shows that
$D$ is demi-normal, and adjunction makes it semi-log-canonical;
see also the normal-total-space argument in the proof of
\cite[Theorem~4.12]{ABP}.
Writing $\nu\colon D^\nu\to D$ for normalization, this gives
\begin{equation}\label{proper:eq:adjunction}
 (D^\nu,\Theta)\text{ is log canonical},\qquad
 \Theta=\Diff_{D^\nu}(0).
\end{equation}
The conductor is included in $\Theta$. In particular, normality of
$D$ has not been assumed.

By Lemma~\ref{proper:lem:normal}, the actual fiber product $X^{[q]}$ is normal,
and $K_{X^{[q]}}+D^{[q]}$ is $\Q$-Cartier. Under the identification $D^{[q]}\cong D$,
the adjunction different on the normalization of the closed fiber
agrees with $\Theta$. This is a compatibility statement about the
different itself, not just a numerical or linear equivalence of
canonical divisors; the argument is recorded in
\cite[proof of Proposition~7.1.2]{BenozzoPosva}.

We may therefore apply three-dimensional inversion of adjunction in
characteristic greater than five
\cite[Corollary~10.1]{BMPSTWW}. If $D^{[q]}$ has several components, apply
it to each prime component, retaining the other components in the
boundary. It follows from \eqref{proper:eq:adjunction} that $(X^{[q]},D^{[q]})$ is
log canonical on an open neighborhood of $D^{[q]}$. The morphism
$X^{[q]}\to C^{[q]}$ is proper, so Lemma~\ref{proper:lem:proper} finishes the proof.
\end{proof}

No assumption that the Cartier index be prime to $p$ is imposed in
this argument. Nor does it require $D$ to be normal, or $K_X+D$ to be
ample. The depth and inversion statements are local; properness is
used only in the final passage from a neighborhood of the closed
fiber to the whole family. Under the hypothesis on the geometric
generic fiber, one can alternatively finish by checking horizontal
valuations, and thus omit properness from this proposition.

This answers the surface version in sufficiently large
characteristic. It does not settle the same question for surfaces
in characteristics $2,3,5$, or for families of higher relative
dimension.

\subsection{The threefold case}

\begin{question}\label{proper:q:threefolds}
Does there exist, in arbitrarily large characteristic, a proper flat
family of relative dimension three with geometrically log canonical
generic fiber for which $(X,D)$ is log canonical but some
$(X^{[q]},D^{[q]})$ is not?
\end{question}

The total space now has dimension four, so the three-dimensional
inversion theorem used in Proposition~\ref{proper:prop:surfaces} does not
apply. The following exclusion is valid in every dimension.

\subsubsection{Locally complete intersection total spaces}

\begin{proposition}\label{proper:prop:lci}
Let $f\colon X\to C$ be a flat family with $X$ normal, and assume
\eqref{proper:eq:assumptions}. If $X$ is locally a complete intersection
over $k$ on a neighborhood of $D$, then $(X^{[q]},D^{[q]})$ is log canonical
for every $q=p^e$.

This holds in every characteristic and every relative dimension.
Properness is not needed for this proposition.
\end{proposition}

\begin{proof}
The fiber product is normal by Lemma~\ref{proper:lem:normal}.
Near its closed fiber it is again locally a complete intersection:
adjoining $s$ and then imposing the monic equation $t-s^q$ preserves
that property.

We use Mather--Jacobian inversion of adjunction, which is valid in
arbitrary characteristic \cite[Theorem~3.23]{IR}. On a normal locally
complete intersection variety, MJ discrepancies agree with ordinary
discrepancies \cite[Remark~2.5(2)]{IR}.

Work locally near $D$. Embed $X$ in a smooth variety $A$, with
codimension $c\geq1$ and ideal $I$, and lift $t$ to a function on $A$.
All these statements may be localized at the curve germ. Put
\[
 B=A\times\A^1_s,\qquad F=t-s^q,\qquad
 J=I\cO_B+(F).
\]
Then $J$ defines $X^{[q]}$ in $B$, with codimension $c+1$.
For an ideal $\mathcal I$, notation such as
$(A,\mathcal I^\lambda)$ refers to the discrepancy expression
$A(v;A)-\lambda v(\mathcal I)$; products of ideals add the
corresponding terms.

Fix a closed subset $W$ of $D$ inside this local chart.
The original log canonicity and MJ inversion give
\begin{equation}\label{proper:eq:mj-first}
 \mld(W;A,I^c(t))\geq0.
\end{equation}
Now apply the same inversion theorem to the smooth divisor
$H=A\times\{0\}\subset B$. The ideal $I^c(F)$ restricts to
$I^c(t)$ on $H$, so \eqref{proper:eq:mj-first} yields
\begin{equation}\label{proper:eq:mj-second}
 \mld(W;B,I^c(F)(s))\geq0.
\end{equation}
We identify $W$ with its copy in $H$ and in $D^{[q]}$.

For every divisorial valuation $v$ over $B$,
\begin{equation}\label{proper:eq:ideal-inequality}
 (c+1)v(J)
 =(c+1)\min\{v(I),v(F)\}
 \leq c\,v(I)+v(F).
\end{equation}
Replacing $I^c(F)$ by $J^{c+1}$ therefore increases the log
discrepancy of every such valuation. Equation~\eqref{proper:eq:mj-second}
implies
\[
 \mld(W;B,J^{c+1}(s))\geq0.
\]
A final application of MJ inversion, this time to $X^{[q]}\subset B$,
gives
\[
 \mld_{\MJ}(W;X^{[q]},(s))\geq0.
\]
Since $X^{[q]}$ is normal and locally a complete intersection here,
this is precisely the ordinary lc inequality for $(X^{[q]},D^{[q]})$.

Taking $W$ to be the center of each vertical divisor checks all
vertical valuations. The generic fiber of $X^{[q]}$ is a field extension
of the original generic fiber, so its log canonicity follows from
geometric log canonicity. This checks the horizontal valuations
as well.
\end{proof}

\begin{corollary}\label{proper:cor:nonlci}
Any counterexample to Question~\ref{proper:q:threefolds} must involve
a non-locally-complete-intersection total space near the relevant
center in the closed fiber.
\end{corollary}

Thus smooth total spaces, hypersurfaces and complete intersections
are excluded. Adding variables to the earlier hypersurface example,
or changing its defining equation while retaining a locally complete
intersection total space, cannot meet the strengthened requirements.

\subsubsection{A standard additive-group quotient attempt}

Consider, for $m\geq1$, the derivation
\begin{equation}\label{proper:eq:derivation}
 \partial_m
 =x_1^2\frac{\partial}{\partial x_1}
  +x_2^2\frac{\partial}{\partial x_2}
  +t^m\frac{\partial}{\partial y}
\end{equation}
on $V=\A^4_{x_1,x_2,y,t}$. Its summands commute and have $p$th
power zero, so it defines an $\alpha_p$-action preserving $t$.
The quotient $Y=V/\alpha_p$ is a natural candidate for a family
of threefolds. It is a special case of the constructions in
\cite[Section~4.1]{PosvaPathological}.

The obstruction can be seen before attempting a compactification.
Blow up the smooth codimension-three center
\[
 Z=(x_1=x_2=t=0)\subset V,
\]
and let $E$ be the exceptional divisor. Since
$\ord_E(t)=1$, the ordinary pair discrepancy is
\[
 a(E;V,V_0)=2-1=1.
\]
On the $t$-chart, write $x_i=tu_i$. The pulled-back derivation is
\[
 t\left(
 u_1^2\frac{\partial}{\partial u_1}
 +u_2^2\frac{\partial}{\partial u_2}
 +t^{m-1}\frac{\partial}{\partial y}
 \right).
\]
Saturation removes one factor of $t$, giving foliation discrepancy
$-1$. The divisor \mbox{$E=(t=0)$} is invariant under the saturated
foliation. Hence its ramification index in the quotient is one.

Let $E_Y$ be the corresponding divisor over $Y$.
Foliation adjunction, or equivalently the discrepancy calculation in
\cite[Section~4.1]{PosvaPathological}, now gives
\begin{equation}\label{proper:eq:quotient-discrepancy}
 a(E_Y;Y,Y_0)=1+(p-1)(-1)=2-p,\qquad
 A(E_Y;Y,Y_0)=3-p.
\end{equation}
For every prime $p\geq5$, the original pair is already non-lc.
The calculation is independent of $m$.

This candidate therefore fails at the starting point. Changing
the exponent $m$ does not make it a counterexample in arbitrarily
large characteristic, and compactification cannot remove the
negative discrepancy already present on this affine open set.
The conclusion concerns this family of quotients; it is not a
nonexistence statement for arbitrary nonlinear $\alpha_p$-actions.

\subsubsection{Four-dimensional plt inversion}

Four-dimensional \emph{plt} inversion of adjunction is known in
characteristic $p>5$ under additional hypotheses. In
\cite[Corollary~4.9]{HW}, the ambient space is $\Q$-factorial,
the pair admits a log resolution, and the coefficients satisfy the
stated standard-coefficient condition. The relevant adjoint pair
on the normalization of the prime boundary must be klt. Its
conclusion is plt near that boundary.

These conditions cannot be silently replaced by the hypotheses of
Question~\ref{proper:q:threefolds}. In particular, an adjoint pair that is
only lc does not meet the klt hypothesis. The cited theorem
therefore does not rule out every non-complete-intersection
four-dimensional candidate.

In dimension four, the three-dimensional depth argument above does
not imply that the closed fiber is $S_2$. Conversely, failure of
$S_2$ is not itself
failure of log canonicity of the ambient fiber pair after base
change. It gives no negative discrepancy on $X^{[q]}$ without a
further argument. The same distinction applies to failures of
properness of a moduli space.

\subsection{A discrepancy criterion for a possible example}

The remaining local problem can be expressed in terms of discrepancies.
For the formula quoted in this section, impose in addition that
$D$ is demi-normal, with separable nodes. This is part of the
hypotheses in \cite[Assumption~7.3.1]{BenozzoPosva}; it is \emph{not} an
automatic consequence of \eqref{proper:eq:assumptions} in dimension four.
When $p\neq2$, the node-separability issue causes no extra
restriction here.

Let $E$ be a vertical prime divisor over $X$, and let $E^{(1)}$
be the corresponding divisor after the first Frobenius base change.
Write
\[
 a_E=a(E;X,D),\qquad
 a'_E=a(E^{(1)};X^{[p]},D^{[p]}).
\]
The divisor is called \emph{wild} if $p\mid\ord_E(t)$.
To describe its wild coefficient, choose a normal birational model
which is regular at the generic point of $E$, with uniformizer $z$,
and write
\[
 t=u z^n,\qquad u\text{ a unit},\quad p\mid n.
\]
Then $\gamma_w(E)$ is the vanishing order of $du$:
\[
 du=z^{\gamma_w(E)}\eta,
\]
where $\eta$ is primitive in the differential module at that
generic point. In particular \mbox{$\gamma_w(E)\geq0$}.

Let $T_{X/C}$ denote the saturated foliation of derivations
annihilating functions from $C$. Invariance of $E$ is understood on
a birational model where $E$ appears, using the induced saturated
foliation. Set
\[
 \delta(E)=
 \begin{cases}
  1,& E\text{ is invariant under }T_{X/C},\\
  p,& E\text{ is not invariant under }T_{X/C}.
 \end{cases}
\]
The discrepancy formula is
\begin{equation}\label{proper:eq:wild-formula}
 a'_E=
 \frac{p\,a_E-(p-1)\gamma_w(E)}{\delta(E)}.
\end{equation}
See \cite[Proposition~7.3.3]{BenozzoPosva}. No assertion about a higher-dimensional
non-demi-normal closed fiber is obtained merely by citing this formula.

A vertical divisor satisfying
\begin{equation}\label{proper:eq:target}
 p\,a_E-(p-1)\gamma_w(E)<-\delta(E)
\end{equation}
would prove that $(X^{[p]},D^{[p]})$ is not lc. In particular, if $E$ is
a wild lc place, so $a_E=-1$, then
\begin{equation}\label{proper:eq:lc-place}
 a'_E=
 \begin{cases}
 -p-(p-1)\gamma_w(E),&
       E\text{ invariant},\\[2pt]
 -1-\dfrac{p-1}{p}\gamma_w(E),&
       E\text{ not invariant}.
 \end{cases}
\end{equation}
\begin{samepage}
Thus either of the following would suffice:
\begin{enumerate}
 \item a wild lc place invariant under $T_{X/C}$;
 \item a wild lc place which is not invariant but has
       $\gamma_w(E)>0$.
\end{enumerate}
\end{samepage}
The fact that a wild divisor has positive $\gamma_w$ is not enough
by itself: its original discrepancy may also be positive.
The inequality \eqref{proper:eq:target}, rather than wildness alone, is
what must be checked.

If such a divisor is found in a local model, a complete answer
to Question~\ref{proper:q:threefolds} must still verify that the original
pair is lc everywhere, that the family is proper, and that its
geometric generic fiber is lc. None of these requirements is
supplied by \eqref{proper:eq:wild-formula} alone.

\subsection{Consequences and the remaining question}

For proper families of surfaces in characteristic $p>5$,
Question~\ref{proper:q:proper} has a negative answer by
Proposition~\ref{proper:prop:surfaces}. In any dimension and characteristic,
Proposition~\ref{proper:prop:lci} excludes locally complete intersection
total spaces under the same assumption on the geometric generic fiber.

For general families of threefolds, the arguments above provide
neither a counterexample nor a theorem excluding all counterexamples.
The $\alpha_p$-quotient family
\eqref{proper:eq:derivation} fails because the original pair is non-lc
when $p\geq5$. Formula~\eqref{proper:eq:target} records a possible
criterion for a different construction; it does not assert that
such a construction exists.

The unresolved part is therefore a proper family of threefolds
with a non-complete-intersection total space near the
closed fiber, satisfying \eqref{proper:eq:assumptions}, together with an
explicit divisor of discrepancy less than $-1$ after Frobenius.
No such family is constructed here.

\clearpage
\pagestyle{references}


\begin{thebibliography}{BMP$^+$23}
\bibitem[ABP24]{ABP}
E.~Arvidsson, F.~Bernasconi and Z.~Patakfalvi,
\emph{On the properness of the moduli space of stable surfaces
over $\mathbb Z[1/30]$},
Moduli \textbf{1} (2024), e3, 1--37.
\href{https://doi.org/10.1112/mod.2024.1}{doi:10.1112/mod.2024.1}.

\bibitem[BP26]{BenozzoPosva}
M.~Benozzo and Q.~Posva,
\emph{Base-change of locally stable families in positive characteristic},
preprint (2026),
\href{https://arxiv.org/abs/2607.22106v2}{arXiv:2607.22106v2}.

\bibitem[BMP$^+$23]{BMPSTWW}
B.~Bhatt, L.~Ma, Z.~Patakfalvi, K.~Schwede, K.~Tucker,
J.~Waldron and J.~Witaszek,
\emph{Globally $+$-regular varieties and the minimal model program
for threefolds in mixed characteristic},
Publications Math\'ematiques de l'IH\'ES \textbf{138} (2023), 69--227.
\href{https://doi.org/10.1007/s10240-023-00140-8}{doi:10.1007/s10240-023-00140-8}.

\bibitem[HW23]{HW}
C.~D.~Hacon and J.~Witaszek,
\emph{On the relative minimal model program for fourfolds in
positive and mixed characteristic},
Forum of Mathematics, Pi \textbf{11} (2023), e10, 1--35.
\href{https://doi.org/10.1017/fmp.2023.6}{doi:10.1017/fmp.2023.6}.

\bibitem[HZ20]{HuZong}
Z.~Hu and R.~Zong,
\emph{On base change of local stability in positive characteristics},
preprint (2020),
\href{https://arxiv.org/abs/2001.04083v1}{arXiv:2001.04083v1}.

\bibitem[IR17]{IR}
S.~Ishii and A.~J.~Reguera,
\emph{Singularities in arbitrary characteristic via jet schemes},
in \emph{Hodge Theory and $L^2$-Analysis} (L.~Ji, ed.),
Advanced Lectures in Mathematics, vol.~39,
International Press, Somerville, MA, and Higher Education Press,
Beijing, 2017, 419--449.
\href{https://arxiv.org/abs/1510.05210v2}{arXiv:1510.05210v2}.

\bibitem[Kol13]{KollarSingularities}
J.~Koll\'ar,
\emph{Singularities of the Minimal Model Program},
Cambridge Tracts in Mathematics, vol.~200,
Cambridge University Press, Cambridge, 2013.
\href{https://doi.org/10.1017/CBO9781139547895}{doi:10.1017/CBO9781139547895}.

\bibitem[Kol23]{KollarFamilies}
J.~Koll\'ar,
\emph{Families of Varieties of General Type},
Cambridge Tracts in Mathematics, vol.~231,
Cambridge University Press, Cambridge, 2023.
\href{https://doi.org/10.1017/9781009346115}{doi:10.1017/9781009346115}.

\bibitem[LMM25]{LMM}
C.~Liedtke, G.~Martin and Y.~Matsumoto,
\emph{Torsors over the rational double points in characteristic $p$},
Ast\'erisque \textbf{461} (2025), 75--152.
\href{https://arxiv.org/abs/2110.03650v2}{arXiv:2110.03650v2}.

\bibitem[Pos25]{PosvaPathological}
Q.~Posva,
\emph{Pathological MMP singularities as $\alpha_p$-quotients},
Forum of Mathematics, Sigma \textbf{13} (2025), e185, 1--29.
\href{https://doi.org/10.1017/fms.2025.10129}{doi:10.1017/fms.2025.10129}.

\bibitem[Pos26a]{PosvaQuotients}
Q.~Posva,
\emph{On the singularities of quotients by $1$-foliations},
Nagoya Mathematical Journal \textbf{261} (2026), e6, 1--41.
\href{https://doi.org/10.1017/nmj.2025.10072}{doi:10.1017/nmj.2025.10072};
\href{https://arxiv.org/abs/2311.16694v5}{arXiv:2311.16694v5}.

\bibitem[Pos26b]{PosvaComparison}
Q.~Posva,
\emph{Local stability of Frobenius and Artin--Schreier base-change:
a comparison}, preprint (2026),
\href{https://arxiv.org/abs/2608.24256v2}{arXiv:2608.24256v2}.

\bibitem[Sta26]{Stacks}
The Stacks Project Authors,
\emph{The Stacks Project}, 2026.
\url{https://stacks.math.columbia.edu}.
\end{thebibliography}
\end{document}